\documentclass[12pt,a4paper,twoside,reqno]{amsart}
\title[Lectures on the Strominger system]{Lectures on the Strominger system}
\author[M. Garcia-Fernandez]{Mario Garcia-Fernandez}

\address{Instituto de Ciencias Matem\'aticas (CSIC-UAM-UC3M-UCM)\\
  Nicol\'as Cabrera 13--15, Cantoblanco\\ 28049 Madrid, Spain}
\email{mario.garcia@icmat.es}

\thanks{This project has received funding from the European Union's Horizon 2020 research and innovation programme under the Marie Sklodowska-Curie grant agreement No 655162.}

\usepackage{hyperref,url}
\usepackage{amssymb,latexsym}
\usepackage{amsmath,amsthm}
\usepackage[latin1]{inputenc}
\usepackage{enumerate}
\usepackage[all]{xy} \CompileMatrices
\usepackage{amscd}

\usepackage{slashed} 

\SelectTips{cm}{12} 
\theoremstyle{plain}
\newtheorem{theorem}{Theorem}[section]

\newtheorem{proposition}[theorem]{Proposition}
\newtheorem{conjecture}[theorem]{Conjecture}
\theoremstyle{definition}
\newtheorem{definition}[theorem]{Definition}
\newtheorem{definition-theorem}[theorem]{Definition-Theorem}
\newtheorem{example}[theorem]{Example}
\theoremstyle{remark}
\newtheorem{remark}[theorem]{Remark}

\newtheorem{question}[theorem]{Question}

\numberwithin{equation}{section} 

\newcommand{\tr}{\operatorname{tr}}
\newcommand{\Id}{\operatorname{Id}}

\newcommand{\End}{\operatorname{End}}

\newcommand{\Ker}{\operatorname{Ker}}
\newcommand{\ad}{\operatorname{ad}}

\newcommand{\Aut}{\operatorname{Aut}}
\newcommand{\dbar}{\bar{\partial}}

\newcommand{\CC}{{\mathbb C}}
\newcommand{\PP}{{\mathbb P}}

\newcommand{\RR}{{\mathbb R}}
\newcommand{\ZZ}{{\mathbb Z}}

\renewcommand{\(}{\left(}
\renewcommand{\)}{\right)}

\newcommand{\Vol}{\operatorname{Vol}}

\newcommand{\surj}{\to\kern-1.8ex\to}

\newcommand{\lto}{\longrightarrow}

\newcommand{\cA}{\mathcal{A}}
\newcommand{\cC}{\mathcal{C}}

\newcommand{\cE}{\mathcal{E}}

\newcommand{\cP}{\mathcal{P}}
\newcommand{\cF}{\mathcal{F}}

\newcommand{\cG}{\mathcal{G}}
\newcommand{\cL}{\mathcal{L}}

\newcommand{\cT}{{\mathcal{T}}}

\newcommand{\Lie}{\operatorname{Lie}}

\newcommand{\LieG}{\operatorname{Lie} \cG}

\def\om{\omega}
\def\Om{\Omega}

\def\Vol{\mathrm{Vol}}

\def\Lie{\mathrm{Lie}}
\def\Diff{\mathrm{Diff}}

\def\Id{\mathrm{Id}}

\def\cA{\mathcal{A}}

\def\cG{\mathcal{G}}

\def\cP{\mathcal{P}}
\def\cT{\mathcal{T}}
\def\cE{\mathcal{E}}

\newcommand{\SU}{\mathrm{SU}}

\newcommand{\SL}{\mathrm{SL}}

\begin{document}

\begin{abstract}
These notes give an introduction to the Strominger system of partial differential equations, and are based on lectures given in September 2015 at the GEOQUANT School, held at the Institute of Mathematical Sciences (ICMAT) in Madrid. We describe the links with the theory of balanced metrics in hermitian geometry, the Hermite-Yang-Mills equations, and its origins in physics, that we illustrate with many examples. We also cover some recent developments in the moduli problem and the interrelation of the Strominger system with generalized geometry, via the cohomological notion of string class.
\end{abstract}

\maketitle

\tableofcontents

\section{Introduction}

The Strominger system of partial differential equations has its origins in supergravity in physics \cite{HullTurin,Strom}, and it was first considered in the mathematics literature in a seminal paper by Li and Yau \cite{LiYau}. The mathematical study of this PDE has been proposed by Yau as a natural generalization of the Calabi problem for non-k\"ahlerian complex manifolds \cite{Yau2005}, and also in relation to 
\emph{Reid's fantasy} on the moduli space of projective Calabi-Yau threefolds \cite{Reid}. There is a conjectural relation between the Strominger system and conformal field theory, which arises in a certain physical limit in compactifications of the heterotic string theory.

In complex dimensions one and two, solutions of the Strominger system are given (after conformal re-scaling) by polystable holomorphic vector bundles and K\"ahler Ricci flat metrics. In dimension three, the existence and uniqueness problem for the Strominger system is still open, and it is the object of much current investigation (see Section \ref{sec:existence}). The existence of solutions has been conjectured by Yau under natural assumptions \cite{Yau2010} (see Conjecture \ref{conj:Yau}).

The main obstacle to prove the existence of solutions in complex dimension three and higher is an intricate equation for 4-forms 
\begin{equation}\label{eq:Bianchiintro}
dd^c\omega = \operatorname{tr} R_\nabla \wedge R_\nabla - \operatorname{tr} F_A \wedge F_A,
\end{equation}
coupling the K\"ahler form $\omega$ of a (conformally) balanced hermitian metric on a complex manifold $X$ with a pair of Hermite-Yang-Mills connections $\nabla$ and $A$.
This subtle condition -- which arises in the quantization of the sigma model for the heterotic string -- was studied by Freed \cite{Freed} and Witten \cite{WittenCMP} in the context of index theory for Dirac operators, and more recently it has appeared in the topological theory of \emph{string structures} \cite{Bunke,Redden,SSS} and in generalized geometry \cite{BarHek,GF,GRT}. Despite these important topological and geometric insights, to the present day we have a very poor understanding of equation \eqref{eq:Bianchiintro} from an analytical point of view.

In close relation to the existence problem, an important object in the theory of the Strominger system is the moduli space of solutions. The moduli problem for the Strominger system is largely unexplored, and only in recent years there has been progress in the understanding of its geometry \cite{AGS,CGT,OssaSvanes,GRT,GRT2}. From a physical perspective, it can be regarded as a first approximation to the moduli space of 2-dimensional (0,2)-superconformal field theories, and is expected to host a generalization of mirror symmetry \cite{Yau2005}.

\subsection*{Organization:}

These lecture notes intend to give an introduction to the theory of the Strominger system, going from classical hermitian geometry to the physical origins of the equations, and its many legs in geometric analysis, algebraic geometry, topology, and generalized geometry. Hopefully, this manuscript also serves as a guide to the vast literature in the topic.

In Section \ref{sec:metrics} we give an introduction to the theory of balanced metrics in hermitian geometry (in the sense of Michelson \cite{Michel}). In Section \ref{sec:balancedilatino} we study balanced metrics in non-k\"ahlerian Calabi-Yau manifolds and introduce the dilatino equation, one of the building blocks of the Strominger system. In Section \ref{sec:HermEin} we go through the theory of Hermite-Einstein metrics on balanced manifolds, and its relation with slope stability. Section \ref{sec:Strominger} is devoted to the definition of the Strominger system and the existence of solutions. In Section \ref{sec:existence} we discuss several methods to find solutions of the equations, that we illustrate with examples, and comment on Yau's Conjecture for the Strominger system (Conjecture \ref{conj:Yau}).

In Section \ref{sec:physics} we review the physical origins of the Strominger system in string theory. This provides an important motivation for its study, and reveals the links of the Strominger system with conformal field theory, and the theory of string structures. For the physical jargon, we refer to the Glossary in \cite{QFSI}. Finally, in Section \ref{sec:moduli}  we consider recent developments in the geometry of the Strominger system, based on joint work of the author with Rubio and Tipler \cite{GRT}. As we will see, the interplay of the Strominger system with the notion of string class \cite{Redden} leads naturally to an interesting relation with Hitchin's theory of generalized geometry \cite{Hit1}, that we discuss in the context of the moduli problem in Section \ref{sec:modulibis}.
\\
\\
\textbf{Acknowledgements:} I would like to thank Bjorn Andreas -- who introduced me to this topic --, Luis \'Alvarez-C\'onsul, Xenia de la Ossa, Antonio Otal, Roberto Rubio, Eirik Svanes, Carl Tipler, and Luis Ugarte for useful discussions and comments about the manuscript.  I thank the organizers of GEOQUANT 2015 for the invitation to give this lecture course, and for their patience with the final version of this manuscript.

\section{Special metrics in hermitian geometry}\label{sec:metrics}

\subsection{K\"ahler, balanced, and Gauduchon metrics}
Let $X$ be a compact complex manifold of dimension $n$, with underlying smooth manifold $M$. A hermitian metric on $X$ is a riemannian metric $g$ on $M$ such that $g(J \cdot ,J\cdot) = g$, where $J \colon TM \to TM$ denotes the almost complex structure determined by $X$. Denote by $\Omega^k$ (resp. $\Omega^k_\CC$) the space of real (resp. complex) smooth $k$-forms  on $M$. Denote by $\Omega^{p,q} \subset \Omega^{p+q}_\CC$ the space of smooth complex $(p+q)$-forms on $X$ of type $(p,q)$. Note that $\Omega^{p,q}$ belongs to the eigenspace of $\Omega^{p+q}_\CC$ with eigenvalue $i^{q-p}$ with respect to the endomorphism
$$
\alpha \to (-1)^{p+q}\alpha(J \cdot,\ldots,J\cdot).
$$
Associated to $g$ there is a canonical non-degenerate $(1,1)$-form $\omega \in \Omega^{1,1}$, defined by
$$
\omega(V,W) = g(JV,W)
$$
for any pair of vector fields $V,W$ on $M$. The $2$-form $\omega$ is called the \emph{K\"ahler form} of the hermitian manifold $(X,g)$.

By integrability of the almost complex structure, we have the decomposition of the exterior differential $d = \partial + \dbar$ acting on $\Omega^{p,q}$, where $\partial$ and $\dbar$ are given by projection
$$
\partial \colon \Omega^{p,q} \to \Omega^{p+1,q}, \qquad \dbar \colon \Omega^{p,q} \to \Omega^{p,q+1}.
$$
Consider the operator $d^c = i(\dbar - \partial)$ acting on forms $\Omega^k_\CC$. We have the following special types of hermitian structures. 

\begin{definition}\label{def:KBG}
A hermitian metric $g$ on $X$ is
\begin{enumerate}[i)]
\item K\"ahler if $d \omega = 0$,
\item balanced if $d \omega^{n-1} = 0$,
\item Gauduchon if $dd^c(\omega^{n-1}) = 0$.
\end{enumerate}
\end{definition}

\begin{remark}
There are other important notions of special hermitian metrics (see e.g. \cite{Fino,Popovici}), such as pluriclosed metrics or astheno-K\"ahler metrics (given by the conditions $dd^c \omega = 0$ and $dd^c(\omega^{n-2}) = 0$, respectively), but their study goes beyond the scope of the present notes.
\end{remark}

The K\"ahler condition for $g$ is equivalent to $\nabla^g J = 0$, where $\nabla^g$ denotes the Levi-Civita connection of the riemannian metric $g$ (see \cite[Section 1.1]{Gauduchon}). Using that $dd^c = - d^cd$, we have a simple chain of implications:

\medskip
\begin{center}
K\"ahler $\Rightarrow$ balanced $\Rightarrow$ Gauduchon.
\end{center}
\medskip

We say that a complex manifold is \emph{k\"ahlerian} (respectively \emph{balanced}) if it admits a K\"ahler (respectively \emph{balanced}) metric. The existence of K\"ahler and balanced metrics in a compact complex manifold is a delicate question \cite{HarLaw,Michel}. In complex dimension two, the conditions of being balanced and k\"ahlerian are both equivalent to the first Betti number of $X$ being even (see e.g. \cite[Th. 1.2.3]{lt}). Thus, there are complex surfaces --- such as the Hopf surfaces --- which carry no K\"ahler metric (see Example \ref{example:CalEck}). However, in all dimensions $n \geqslant 3$ there exist compact balanced manifolds which are not k\"ahlerian. This is true, for example, for certain complex nilmanifolds (see Example \ref{example:balanced}). In contrast, due to a theorem of Gauduchon \cite{Gau0} every compact complex manifold admits a Gauduchon metric.

\begin{theorem}[\cite{Gau0}]
Every hermitian metric on $X$ is conformal to a Gauduchon metric, uniquely up to scaling when $n > 1$.
\end{theorem}

A large class of k\"ahlerian complex manifolds is given by the projective algebraic manifolds. This follows from the basic fact that the k\"ahlerian property is inherited by holomorphic immersions, that is, if $X$ is k\"ahlerian and there exists a
holomorphic immersion $f \colon Y \to X$, then $Y$ is k\"ahlerian (by pull-back of a K\"ahler metric on $X$).

\begin{example}\label{example:Kahler}
Any closed complex submanifold of $\mathbb{CP}^N$ is k\"ahlerian. Recall that any closed analytic submanifold $X \subset \mathbb{CP}^N$ is algebraic, by Chow's Theorem.
\end{example}

Basic examples of balanced manifolds which are not k\"ahlerian can be found among complex parallelizable manifolds. Compact complex parallelizable manifolds were characterized by Wang \cite{Wang}, who proved that all arise as a quotient of a complex unimodular Lie group $G$ by a discrete subgroup $\Gamma$. They are, in general, non-k\"ahlerian: such a manifold is k\"ahlerian if and only if it is a torus \cite[p.776]{Wang}. Using Wang's characterization, Abbena and Grassi~\cite{AbGr} showed that all parallelizable complex manifolds are balanced. In fact, any right invariant hermitian metric on $G$ is balanced and this induces a balanced metric on the manifold $G/\Gamma$ \cite[Theorem 3.5]{AbGr}. We discuss a concrete example on the Iwasawa manifold, due to Gray (see e.g. \cite[p. 120]{Gau0bis}).

\begin{example}\label{example:balanced}
Let $G \subset \SL(3,\CC)$ be the non-abelian group given by elements of the form
\begin{equation}
\( \begin{array}{ccc}
1 & z_1 & z_2 \\
0 & 1 & z_3 \\
0 & 0 & 1
\end{array} \) \in \SL(3,\CC),
\end{equation}
for $z_1,z_2,z_3$ arbitrary complex numbers. Let $\Gamma  \subset G$ be the subgroup whose entries are given by Gaussian numbers
$$
\Gamma = G \cap \SL(3,\ZZ[i]).
$$
The quotient $X = G/\Gamma$ is a compact complex parallelizable manifold of dimension $3$. The holomorphic cotangent bundle $T^*X$ can be trivialized explicitly in terms of three holomorphic $1$-forms $\theta_1$, $\theta_2$, $\theta_3$, locally given by
$$
\theta_1 = dz_1, \qquad \theta_2 = dz_2 - z_3 dz_1, \qquad \theta_3 = dz_3,
$$
satisfying the relations
$$
d \theta_1 = d\theta_3 = 0, \qquad d \theta_2 = \theta_1 \wedge \theta_3.
$$
Since $d \theta_2 = \partial \theta_2$ is a non-vanishing exact holomorphic $(2,0)$-form, $X$ is not k\"ahlerian (by the $\partial\dbar$-lemma, any such form vanishes in a K\"ahler manifold). To show that $X$ is balanced, we can take the positive $(1,1)$-form 
\begin{equation}\label{eq:omegaparallel}
\omega = i (\theta_1 \wedge \overline{\theta}_1 + \theta_2 \wedge \overline{\theta}_2 + \theta_3 \wedge \overline{\theta}_3),
\end{equation}
which defines a hermitian metric $g = \omega(\cdot , J \cdot)$ on $X$. It can be readily checked that $d \omega \wedge \omega = 0$.
\end{example}

We introduce next an important $1$-form canonically associated to any hermitian structure, that we use to give a characterization of the balanced and Gauduchon conditions..

\begin{definition}\label{def:Lee}
The \emph{Lee form} of a hermitian metric $g$ on $X$ is the $1$-form $\theta_\omega \in \Omega^1$ defined by
\begin{equation}\label{eq:Lee}
\theta_\omega = J d^*\omega.
\end{equation}
\end{definition}

Here, $d^* = - *d*$ is the adjoint of the exterior differential $d$ for the hermitian metric $g$, where $*$ denotes the (riemannian) \emph{Hodge star operator} of $g$. Alternatively, using the operator
\begin{equation}
\label{eq:Lambda}
 \Lambda_\omega :\Omega^k\lto \Omega^{k-2}\colon \psi \longmapsto \iota_{\omega^{\sharp}} (\psi),
\end{equation}
where $\sharp$ is the operator acting on $k$-forms induced by the
symplectic duality $\sharp\colon T^*X \to TX$ and $\iota$ denotes
the contraction operator, we have
$$
\theta_\omega = \Lambda_\omega d\omega.
$$
A different way of defining the Lee form is via the equation
$$
d \omega^{n-1} = \theta_\omega \wedge \omega^{n-1},
$$
which in particular implies $\Lambda_\omega d\theta_\omega = 0$. From the previous formula we can deduce the change of the Lee form under conformal transformations: if $\tilde \omega = e^\phi \omega$ for a smooth function $\phi \in C^\infty(X)$ then
\begin{equation}\label{eq:Leeconf}
\theta_{\tilde \om} = \theta_\om + (n-1)d\phi.
\end{equation} 
Note that $d\theta_\omega$ is a conformal invariant (when $\theta_\om$ is closed, so is $[\theta_\omega] \in H^1(X,\RR)$). 

\begin{proposition}\label{prop:Lee}
A hermitian metric $g$ on $X$ is
\begin{enumerate}[i)]
\item balanced if and only if $\theta_\omega = 0$,
\item Gauduchon if and only if $d^*\theta_\om = 0$.
\end{enumerate}
\end{proposition}
\begin{proof}
The first part follows from the equality $* \omega = \frac{\omega^{n-1}}{(n-1)!}$, since by definition $d^* = - *d*$. As for the second part, the statement follows from
$$
d^*\theta_\om = - * dJd * \omega  = - * J d^cd * \omega, 
$$
where we have used that $d^c = JdJ^{-1}$, and the algebraic identities $* J = J*$, $*^2 \alpha^k = (-1)^k \alpha^k$ for any $\alpha^k \in \Omega^k$.
\end{proof}

\subsection{Balanced manifolds}\label{sec:balanced}

In this section we study general properties of balanced manifolds, which enable to construct a large class of examples, and also to identify complex manifolds which are not balanced. 
The guiding principle is that the balanced property for a complex manifold is, in a sense, dual to the k\"ahlerian property. This can be readily observed from Proposition \ref{prop:Lee}, which implies that the balanced condition for a hermitian metric is equivalent to the K\"ahler form being co-closed
$$
d^*\omega = 0.
$$

As mentioned in Example \ref{example:Kahler}, the k\"ahlerian property of complex manifolds is inherited by holomorphic immersions. Balanced manifolds satisfy a dual `functorial property', in terms of proper holomorphic submersions. Thus, the K\"ahler property is induced on sub-objects and the balanced property projects to quotient objects.

\begin{proposition}[\cite{Michel}]\label{prop:submer}
Let $X$ and $Y$ be compact complex manifolds. Then
\begin{enumerate}[i)]
\item If $X$ and $Y$ are balanced, then $X \times Y$ is balanced.
\item Let $f \colon X \to Y$ be a proper holomorphic submersion. If $X$ is balanced then $Y$ is balanced.
\end{enumerate}
\end{proposition}

\begin{proof}
Let $n$ and $m$ denote the complex dimensions of $X$ and $Y$, respectively. To prove $i)$ we simply note that if $\omega_X$ and $\omega_Y$ are (the K\"ahler forms of) balanced metrics on $X$ and $Y$, respectively, then
$$
(\omega_X + \omega_Y)^{n + m - 1} =   \binom{n-1}{m+n-1}\omega_X^{n-1} \wedge \omega_Y^m  + \binom{m-1}{m+n-1} \omega_X^n \wedge \omega_Y^{m-1},
$$
for the product hermitian structure on $X \times Y$. 

As for $ii)$, let $\omega_X$ be a balanced metric on $X$ and consider the closed $(n-1,n-1)$-form $\tau_X = \omega_X^{n-1}$. Since $f$ is proper, there exists a closed $2m-2$-form $\tau_Y = f_*\tau_X$ on $Y$, given by integration along the fibres. The fibres are complex, and hence $\tau_Y$ is actually a form of type $(m-1,m-1)$. Furthermore, it can be checked that it is positive, in the sense that it induces a positive form on every complex hyperplane in $TY$. A linear algebra argument (see \cite[p. 280]{Michel}) shows now that $\tau_Y$ admits an $(m-1)$-th positive root, that is, a positive $(1,1)$-form $\omega_Y$ on $Y$, such that $\omega_Y^{m-1} = \tau_Y$.  
\end{proof}

The converse of $ii)$ in Proposition \ref{prop:submer} is not true. To see a counterexample, we need a basic cohomological property of balanced manifolds. This shall be compared with the fact that, on a k\"alerian manifold, no compact complex curve can be trivial in homology.

\begin{proposition}[\cite{Michel}]\label{prop:cohomol}
Let $X$ be a compact complex balanced manifold of dimension $n$. Then every compact complex subvariety of codimension $1$ represents a non-zero class in $H_{2n-2}(X,\RR)$.
\end{proposition}

The proof of Proposition \ref{prop:cohomol} follows easily by integration of the closed form $\omega^{n-1}$ along the subvariety, for any balanced metric on $X$. This property of codimension $1$ subvarieties of a balanced manifold was strengthened by Michelson in terms of currents, whereby he provided a cohomological characterization of those compact complex manifolds which admit balanced metrics \cite[Th. A]{Michel}. We will not comment further on this more sophisticated notion, since Proposition \ref{prop:cohomol} is enough to give first examples of complex manifolds which are not balanced.

\begin{example}[\cite{CalabiEck}]\label{example:CalEck}
Consider $\hat X = \CC^{p+1} \backslash \{0\} \times \CC^{q+1} \backslash \{0\}$ and the abelian group $\CC$ acting on $\hat X$ by
$$
t\cdot (x,y) = (e^t x,e^{\alpha t} y),
$$
for $\alpha \in \CC \backslash \RR$. The \emph{Calabi-Eckmann manifold} $X = \hat X /\CC$, is diffeomorphic to $S^{2p + 1} \times S^{2q + 1}$, and admits a natural holomorphic fibration structure
$$
f \colon X \to \CC\PP^p \times \CC\PP^q
$$
induced by the product of the Hopf mappings $S^{2k + 1} \to \CC\PP^k$, for $k = p,q$. Hence, $X$ admits plenty of compact complex submanifolds of codimension-one. Since the homology is zero in dimension $2p + 2q$, we see that these manifolds support no balanced metrics if $p + q > 0$. Note that $p$ is proper and $\CC\PP^p \times \CC\PP^q$ is k\"ahlerian, and hence the converse of $ii)$ in Proposition \ref{prop:submer} does not hold.
\end{example}

The next result, also due to Michelson, provides a weak converse of $ii)$ in Proposition \ref{prop:submer}. 

\begin{theorem}[\cite{Michel}]\label{th:fibr}
Let $X$ be a compact complex connected manifold. Suppose that $X$ admits a holomorphic map $f \colon X \to C$ onto a complex curve $C$, with a cross section. If the non-singular fibres of $f$ are balanced, then $X$ is balanced.
\end{theorem}

Note that the existence of the cross section implies that the pull-back $f^*[C]$ of the fundamental class $[C]$ of $C$ is non-trivial in $H^2(X,\RR)$, and therefore no regular fibre can be homologous to zero (since its class in homology is the Poincare dual of $[C]$). To illustrate this result, consider the Hopf surface $S^1 \times S^3$ -- regarded as a Calabi-Eckmann manifold in Example \ref{example:CalEck} --. This manifold admits a holomorphic submersion $S^1\times S^3 \to \CC\PP^1$ with K\"ahler fibres (given by elliptic curves) but it is not balanced, because the fibres are homologically trivial.

The following example, which fulfils the hypothesis of Theorem \ref{th:fibr}, is due to Calabi.

\begin{example}[\cite{Calabi}]\label{example:Calabi}
Let $C$ be a compact Riemann surface and $\iota \colon \tilde C \to \RR^3$ a conformal minimal immersion, where $\tilde C$ denotes the universal cover of $C$. Considering the product immersion
$$
\iota \times \Id \colon \tilde C \times \RR^4 \to \RR^3 \oplus \RR^4 = \RR^7,
$$
the $6$-dimensional submanifold $\iota(\tilde C \times \RR^4)$ inherits an integrable complex structure induced by Cayley multiplication (where $\RR^7$ is regarded as the imaginary octonions). This complex structure is invariant under covering transformations on $\tilde C$ and translations on $\RR^4$, and satisfies that the natural inclusions $\tilde C \times \{x\} \to \tilde C \times \RR^4$ are holomorphic for any $x \in \RR^4$. Hence for any lattice $\Lambda \subset \RR^4$ we can produce a compact quotient manifold $X_\Lambda$ which admits a holomorphic projection $f \colon X_\Lambda \to C$. The fibres of this map are complex tori and the map $f$ has holomorphic cross sections (induced by the inclusions $\tilde C \times \{x\} \to \tilde C \times \RR^4$). By Theorem \ref{th:fibr}, $X_\Lambda$ is a balanced manifold. Calabi proved in \cite{Calabi} that the manifolds $X_\Lambda$ are not k\"ahlerian.
\end{example}

Further examples of balanced holomorphic fibrations (which do not satisfy the hypothesis of Theorem \ref{th:fibr}) are given by twistor spaces of four-dimensional riemannian manifolds. Given an oriented riemannian $4$-manifold $N$, there is associated twistor space $T$, given by an $S^2$-bundle over $N$. The fibre of $T$ over a point $x$ in $N$ is the sphere of all orthogonal almost complex structure on $T_xN$ compatible with the orientation. The twistor space $T$ has a canonical almost complex structure, which is integrable if
and only if $N$ is self-dual \cite{AHS}. Furthermore, there is a natural balanced metric on $T$. As shown by Hitchin \cite{HitchinTwistor}, the only compact twistor spaces which are K\"ahler are those associated $S^4$ and $\CC\PP^2$.

The next result is due to Alessandrini and Bassanelli, and proves that the property of being balanced is a birational invariant. Due to a counterexample of Hironaka \cite{Hiro}, this is not true for K\"ahler manifolds, and thus the existence of balanced metrics is a more robust property than the k\"ahlerian condition for a compact complex manifold.

\begin{theorem}[\cite{AlBas1,AlBas2}]\label{th:birational}
Suppose $X$ and $Y$ are compact complex manifolds. Let $f \colon X \to Y$ be a modification. Then $X$ is balanced if and only if $Y$ is balanced.
\end{theorem}

A modification $f \colon X \to Y$ is a holomorphic map such that there exists a complex submanifold $N \subset Y$ of codimension at least two and a biholomorphism $f \colon X \backslash f^{-1}(N) \to Y \backslash N$ given by restriction. As a corollary of the previous result, any compact complex manifold of Fujiki class $\mathcal{C}$ is balanced, since, by definition, it is bimeromorphic to a K\"ahler manifold.

We finish this section with some comments on the behaviour of the k\"ahlerian and balanced properties under deformations of complex structure. Kodaira and Spencer proved that any small deformation of a compact K\"ahler manifold is again a K\"ahler manifold \cite[Th. 15]{KodSp}. Unlike for k\"ahlerian manifolds, the existence of balanced metrics on a compact complex manifold is not an open condition under small deformations of the complex structure. This was shown explicitly in  \cite[Proposition 4.1]{AlBas0}, for the Iwasawa manifold endowed with the holomorphically parallelizable complex structure (see Example \ref{example:balanced}). As shown in \cite{AngUg,FuYaubalanced,Wu}, a balanced analogue of the stability result of Kodaira and Spencer 
requires a further assumption on the variation of Bott-Chern cohomology of the complex manifold. In particular, Wu proved in \cite[Th. 5.13]{Wu} that small deformations of compact balanced manifolds satisfying the $\partial \dbar$-lemma still admit balanced metrics.

Hironaka's Example \cite{Hiro} mentioned above shows in particular that the K\"ahler property is not closed under deformations of complex structure. 
Recently, Ceballos, Otal, Ugarte and Villacampa \cite{COUV} proved the analogue result for balanced manifolds, that is, that the balanced property is not closed under holomorphic deformations. This result has been strengthened in \cite{FiOU}, via the construction of a family over a disk whose generic element is a balanced manifold satisfying the $\partial \dbar$-lemma and whose central fibre is not balanced.

\section{Balanced metrics on Calabi-Yau manifolds}\label{sec:balancedilatino}

\subsection{The dilatino equation}\label{sec:dilatino}

We introduce next an equation -- for a hermitian metric on a complex manifold with trivial canonical bundle -- which is closely related to the balanced condition, and constitutes one of the building blocks of the Strominger system. We will use the following notation throughout.

\begin{definition}\label{def:CY}
A Calabi-Yau $n$-fold is a pair $(X,\Omega)$, given by a complex manifold $X$ of dimension $n$ and a non-vanishing holomorphic global section $\Omega$ of the canonical bundle $K_X = \Lambda^n T^*X$.
\end{definition}

We should stress that in the previous definition we do not require $X$ to be k\"ahlerian. To introduce the dilatino equation, given a hermitian metric $g$ on $X$ we will denote by $\|\Omega\|_\omega$ the norm of $\Omega$, given explicitly by
\begin{equation}\label{eq:normnform}
\|\Omega\|_\omega^2 \frac{\omega^n}{n!} = (-1)^{\frac{n(n-1)}{2}}i^n \Omega \wedge \overline{\Omega}.
\end{equation}

\begin{definition}\label{def:dilatino}
The \emph{dilatino equation}, for a hermitian metric $g$ on $(X,\Omega)$, is
\begin{equation}\label{eq:dilatinosec2}
d^*\omega = d^c \log \|\Omega\|_\omega.
\end{equation} 
\end{definition}

When $n = 1$, $g$ is necessarily K\"ahler and the Calabi-Yau assumption implies that $X$ is an elliptic curve. Then, \eqref{eq:dilatinosec2} reduces to $d \log \|\Omega\|_\omega = 0$, which is equivalent to $\omega$ being Ricci-flat. We assume $n>1$ in the sequel. 

The next result shows that the existence of solutions of \eqref{eq:dilatinosec2} is equivalent to the existence of a K\"ahler Ricci-flat metric, if $n=2$, and to the existence of a balanced metric when $n \geqslant 3$. 
The proof relies on the observation 
in \cite{GMW,LiYau} that the dilatino equation is equivalent to the \emph{conformally balanced equation} 
\begin{equation}\label{eq:LiYau}
d (\|\Omega\|_\omega \omega^{n-1}) = 0.
\end{equation}

\begin{proposition}\label{prop:conformalbalanced}
Let $\sigma$ be a hermitian conformal class on $(X,\Omega)$. Then
\begin{enumerate}[i)]
\item if $n = 2$, then $\sigma$ admits a solution of \eqref{eq:dilatinosec2} if and only if all $g \in \sigma$ is a solution, if and only there exists a K\"ahler Ricci-flat metric on $\sigma$.
\item if $n \geqslant 3$ then $\sigma$ admits a solution of \eqref{eq:dilatinosec2} if and only if $\sigma$ admits a balanced metric.
\end{enumerate}
If $X$ is compact, then there exists at most one balanced metric on $\sigma$ up to homothety.
\end{proposition}

\begin{proof}
We sketch the proof and leave the details for the reader. First we use that \eqref{eq:dilatinosec2} is equivalent to
$$
\theta_\omega = - d \log \|\Omega\|_\omega,
$$
(in particular $\theta_\omega$ needs to be exact if there exists a solution). Thus, using \eqref{eq:Leeconf} the dilatino equation is equivalent to \eqref{eq:LiYau}, which holds if and only if $\tilde \omega = \|\Omega\|_\omega^{\frac{1}{n-1}} \omega$ is balanced. Conversely, for $n \geqslant 3$, $\tilde \omega$ is balanced if and only if $\omega =  \|\Omega\|_{\tilde \omega}^{\frac{-2}{n-2}} \tilde \omega$ solves the dilatino equation. The existence part of the statement follows from the change of the $(n-1,n-1)$-form $\|\Omega\|_\omega \omega^{n-1}$ under a conformal transformation of the metric. The uniqueness part follows by direct application of Gauduchon's Theorem \cite[Th. I.14]{Gau1984}. 
\end{proof}

The equivalence between \eqref{eq:dilatinosec2} and \eqref{eq:LiYau} implies that any solution of the dilatino equation has an associated class in cohomology. To be more precise, consider the Bott-Chern cohomology of $X$, defined by
$$
H^{p,q}_{BC}(X) = \frac{\Ker (d \colon \Omega^{p,q} \to \Omega^{p+1,q} \oplus \Omega^{p,q+1})}{\operatorname{Im} (dd^c \colon \Omega^{p-1,q-1} \to \Omega^{p,q})}.
$$
Since $d$ and $dd^c$ are real operators, the cohomology groups of bi-degree $(p,p)$ have a natural real structure
$$
H^{p,p}_{BC}(X,\RR) \subset H^{p,p}_{BC}(X).
$$

\begin{definition}\label{def:balancedclass}
Given a solution $\omega$ of the dilatino equation \eqref{eq:dilatinosec2}, its \emph{balanced class} is defined by
\begin{equation}\label{eq:balancedclass}
[\|\Omega\|_\omega \omega^{n-1}] \in H^{n-1,n-1}_{BC}(X,\RR).
\end{equation}
\end{definition}

\begin{remark}\label{rem:balanceddef}
Given a balanced metric $\tilde \omega$, there are infinitely many such metrics with the same balanced class $[\tilde \omega^{n-1}]$, as we can always take a real form $\varphi \in \Omega^{n-2,n-2}$ and deform $\tilde \omega^{n-1}$ by
$$
\Psi = \tilde \omega^{n-1} + dd^c \varphi.
$$
For $\varphi$ small we have that $\Psi$ is positive, and thus its $(n-1)$-th root is a balanced metric \cite{Michel}.
\end{remark}

A solution of the dilatino equation has an alternative interpretation, in terms of a connection with skew-torsion and restricted holonomy. Let $X$ be a complex manifold, with underlying smooth manifold $M$. 
A hermitian connection on $(X,g)$ is a linear connection on $TM$ such that $\nabla J = 0$ and $\nabla g = 0$, where $J$ is the almost complex on $M$ determined by $X$. Gauduchon observed in \cite{Gau1} that the \emph{Bismut connection} \cite{Bismut}
\begin{equation}\label{eq:Bismut}
\nabla^B = \nabla^g - \frac{1}{2}g^{-1}d^c\omega.
\end{equation}
is the unique hermitian connection on $X$ with skew-symmetric torsion. Here we regard the metric as an isomorphism $g \colon TM \to T^*M$ and $g^{-1}d^c\omega$ as $1$-form with values on the endomorphisms of $TM$. Note that the torsion of $\nabla^B$ is given by
$$
g T_{\nabla^B} = - d^c\omega \in \Omega^3.
$$

\begin{proposition}\label{prop:dilatino}
Let $(X,\Omega)$ be a Calabi-Yau manifold endowed with a hermitian metric $g$. If $g$ is a solution of the dilatino equation \eqref{eq:dilatinosec2}, then the Bismut connection $\nabla^B$ has restricted holonomy in the special unitary group
\begin{equation}\label{eq:hol}
hol(\nabla^B) \subset SU(n).
\end{equation}
The converse is true if $X$ is compact.
\end{proposition}


For the proof, we need a formula for the unitary connection induced by $\nabla^B$ on $K_X$ due to Gauduchon \cite[Eq. (2.7.6)]{Gau1}. Let $\nabla^C$ be the Chern connection of the hermitian metric on $K_X$ induced by $g$. Then,
\begin{equation}\label{eq:nablaC}
  \nabla^B = \nabla^C - i d^*\omega \otimes \Id.
\end{equation}
Recall that $\nabla^C$ is uniquely determined by the property $(\nabla^C)^{0,1} = \dbar$, where $\dbar$ is the canonical Dolbeault operator on the holomorphic line bundle $K_X$. Since $X$ is Calabi-Yau, we have an explicit formula for $\nabla^C$ in \eqref{eq:nablaC} in the holomorphic trivialization of $K_X$ given by $\Omega$, namely,
$$
\nabla^C = d + 2 \partial \log \|\Omega\|_\omega.
$$

\begin{proof}[Proof of Proposition \ref{prop:dilatino}]
Let $\psi$ be a smooth section of $K_X$. Then, there exists $f \in C^\infty(X,\CC)$ a smooth complex-valued function on $X$ such that $\psi = e^f \Omega$. Applying \eqref{eq:nablaC} we have
\begin{equation}\label{eq:nablaBpsi}
\nabla^B(\psi) = (df + 2 \partial \log \|\Omega\|_\omega - i d^*\omega) \otimes \psi.
\end{equation}
Assume that $\omega$ is a solution of \eqref{eq:dilatinosec2}. Then, by the previous equation
$$
\nabla^B(\psi) = (df + d \log \|\Omega\|_\omega) \otimes \psi
$$
and we can set $f = - \log \|\Omega\|_\omega$ to obtain a parallel section $\psi =  \|\Omega\|_\omega^{-1} \Omega$.

For the converse, if $\psi$ is parallel with respect to $\nabla^B$, applying \eqref{eq:nablaBpsi} we obtain
$$
d f = i d^*\omega - 2 \partial \log \|\Omega\|_\omega.
$$
Further, $\psi$ must have constant norm $\|\psi\|_\omega = t \in \RR_{>0}$ and it follows from \eqref{eq:normnform} that
$$
f + \overline{f} = 2\log t - 2\log \|\Omega\|_\omega.
$$
Hence, setting $\phi = f - \overline{f}$ we obtain
$$
d\phi = d^*\omega - d^c\log \|\Omega\|_\omega.
$$
It suffices to prove that $\phi$ is constant. For this, we define $\tilde \omega = \|\Omega\|_\om^{\frac{1}{n-1}}\omega$ and note that
$$
d\phi = - J \theta_{\tilde \omega} = d^{\tilde *} \tilde \omega,
$$
by the behaviour of the Lee under conformal rescaling \eqref{eq:Leeconf}. Applying now the operator $d^{\tilde *}$ on both sides of the equation $\Delta_{\tilde \om} \phi = d^{\tilde *}d\phi = 0$, and therefore $d \phi = 0$ since by assumption $X$ is compact.
\end{proof}

As we have just seen, the restriction of the holonomy of the Bismut connection to the special unitary group on a Calabi-Yau manifold is essentially equivalent to the existence of solutions of the dilatino equation \eqref{eq:dilatinosec2}. This equation for the K\"ahler form of the hermitian structure is strongly reminiscent of the complex Monge-Amp\`ere equation in K\"ahler geometry. To see this, assume for a moment that $X$ is a domain in $\CC^n$, $\Omega = dz_1 \wedge \ldots \wedge dz_n$ and that $g$ is K\"ahler, with
$$
\omega = dd^c \varphi
$$
for a smooth function $\varphi$. Then $d^*\omega = 0$ and the dilatino equation reduces to the complex Monge-Amp\`ere equation
$$
\log \det \partial_i\partial_{\overline{j}}\varphi = t,
$$
for a choice of constant $t$.
Relying on Yau's solution of the Calabi Conjecture \cite{Yau1977}, the previous observation gives a hint that \emph{Calabi-Yau metrics}, that is, K\"ahler metrics with vanishing Ricci tensor, provide examples of solutions for the dilatino equation.

\begin{example}\label{ex:CY}
Assume that $X$ is compact and k\"ahlerian. Then, by Yau's Theorem \cite{Yau1977}, $X$ admits a unique K\"ahler Ricci-flat metric $g$ on each K\"ahler class. Since $g$ is K\"ahler, 
the left hand side of \eqref{eq:dilatinosec2} vanishes. Further, Ricci-flatness implies that the holomorphic volume form $\Omega$ has to be parallel \cite[Prop. 1.22.6]{Gauduchon}, and therefore $g$ solves \eqref{eq:dilatinosec2}. Note that in this case $d^c \omega = 0$ and hence $\nabla^B = \nabla^C = \nabla^g$. 
\end{example}

To find non-K\"ahler examples of solutions of the dilatino equation we revisit the complex parallelizable manifold in Example \ref{example:balanced}.

\begin{example}\label{example:balanced2}
The compact complex manifold $X = G/\Gamma$ in Example \ref{example:balanced} is parallelizable, and hence admits a trivialization of the canonical bundle. Using the frame $\theta_1, \theta_2, \theta_3$ of $T^*X$ we consider
$$
\Omega = \theta_1\wedge \theta_2 \wedge \theta_3.
$$
Then, it can be readily checked that $\|\Omega\|_\omega$ is constant, and hence the balanced metric \eqref{eq:omegaparallel} is a solution of the dilatino equation \eqref{eq:dilatinosec2}. We note that the same exact argument works in an arbitrary complex parallelizable manifold.
\end{example}

\subsection{Balanced metrics and conifold transitions}\label{sec:conifoldtr}

We describe next a more sophisticated construction of non-k\"ahlerian Calabi-Yau threefolds due to Clemens \cite{Clemens} and Friedman \cite{Friedman} (for a review see \cite{Rossi}), which provides a strong motivation for the study of the Strominger system. These threefolds are typically not birational to k\"ahlerian manifolds (class $\mathcal{C}$) as they may have vanishing second Betti number. 
Thus, Alessandrini-Bassanelli Theorem \ref{th:birational}  cannot be applied to ensure the existence of balanced metrics. An interesting result due to Fu, Li and Yau \cite{FuLiYau} shows that the Calabi-Yau manifolds obtained via the Clemens-Friedman construction are balanced. 

Let $X$ be a smooth k\"ahlerian Calabi-Yau threefold with a collection of mutually disjoint smooth rational curves $C_1, \ldots, C_k$, with normal bundles isomorphic to
$$
\mathcal{O}_{\CC\PP^1}(-1) \oplus \mathcal{O}_{\CC\PP^1}(-1).
$$
Contracting the $k$ rational curves, we obtain a singular Calabi-Yau threefold $X_0$ with $k$ ordinary double-point singularities $p_1, \ldots, p_k$. Away from the singularities, we have a biholomorphism $X \backslash \bigcup_k C_k \cong X_0 \backslash \{p_1, \ldots, p_k\}$, while a neighbourhood of $p_j$ in $X_0$ is isomorphic to a neighbourhood of $0$ in
$$
\{ z_1^2 + z_2^2 + z_3^2 + z_4^2 = 0\} \subset \CC^4.
$$
By results of Friedman, Tian and Kawamata, if the fundamental classes $[C_j] \in H^{2,2}(X,\mathbb{Q})$ satisfy a relation
$$
\sum_j n_j[C_j] = 0,
$$
with $n_j \neq 0$ for every $j$, then there exists a family of complex manifolds $X_t$ over a disk $\Delta \subset \CC$, such that $X_t$ is a smooth Calabi-Yau threefold for $t\neq 0$, and the central fibre is isomorphic to $X_0$. For small values of $t$, the local model for this \emph{smoothing} 
is isomorphic to a neighbourhood of $0$ in
$$
\{ z_1^2 + z_2^2 + z_3^2 + z_4^2 = t \} \subset \CC^4.
$$

The construction of $X_t$ from $X$, typically denoted by a diagram of the form
$$
X \to X_0 \dashrightarrow X_t,
$$
is known in physics as \emph{Clemens-Friedman conifold transition} (see e.g. \cite{Rossi}). The smooth manifolds $X_t$ satisfy the $\partial \dbar$-lemma, but they are in general non-k\"ahlerian \cite{Friedman2}. Explicitly, Friedman observed that $\sharp_k(S^3 \times S^3)$  for any $k \geqslant 2$ can be given a complex structure in this way. The idea is to contract enough rational curves on $X$ so that $H^2(X_t,\RR) = 0$ for $t \neq 0$. 

\begin{example}[\cite{CGH}]\label{ex:Candelas}
Consider the complete intersection $X \subset \CC\PP^4 \times \CC\PP^1$ given by
\begin{align*}
(x_2^4 + x_4^4 - x_5^4)y_1 + (x_1^4 + x_3^4 + x_5^4)y_2 & = 0,\\
x_1 y_1 + x_2 y_2 & = 0,
\end{align*}
where $x_1, \ldots, x_5$ and $y_1,y_2$ are coordinates in $\CC\PP^4$ and $\CC\PP^1$, respectively. This is a Calabi-Yau threefold with $b_2(X) = 2$ (see e.g. \cite{Rossi}). Consider the blow up
$$
\psi \colon \hat \CC\PP^4 \to \CC\PP^4
$$ 
of $\CC\PP^4$ along the plane $\{x_1 = x_2 = 0 \}$, whose exceptional divisor is a $\CC\PP^1$-bundle over $\CC\PP^2$. Then, $X$ can be regarded as the proper transform of the singular hypersurface $X_0 \subset \CC\PP^4$ containing the given plane, defined by
$$
x_2 (x_2^4 + x_4^4 - x_5^4) - x_1 (x_1^4 + x_3^4 + x_5^4) = 0.
$$
Note that $X_0$ has $16$ ordinary double points $p_1, \ldots, p_{16}$, described by
$$
x_1 = x_2 = 0, \qquad  x_3^4 + x_5^4 = 0, \qquad  x_4^4 - x_5^4 = 0,
$$
and $\psi^{-1}(p_j) \cong \CC\PP^1$ for all $j =1, \ldots, 16$. Choosing now a small $t \in \CC$, we can consider the smooth quintic hypersurface $X_t \subset \CC\PP^4$ with equation
$$
x_2 (x_2^4 + x_4^4 - x_5^4) - x_1 (x_1^4 + x_3^4 + x_5^4) = t \sum_{i=1}^5 x_i^5,
$$
which defines a smoothing of $X_0$ (the so called \emph{generic quintic}). Note that $X_t$ is also a Calabi-Yau threefold, and we have decreased the Betti number $b_2(X_t) = 1$.
\end{example}

The main result in \cite{FuLiYau} states that the smoothing $X_t$ in a conifold transition admits a balanced metric, providing first examples of balanced metrics on the complex manifolds $\sharp_k(S^3 \times S^3)$.

\begin{theorem}[\cite{FuLiYau}]
For sufficiently small $t \neq 0$, $X_t$ admits a smooth balanced metric.
\end{theorem}

Reid speculated \cite{Reid} that all k\"ahlerian Calabi-Yau threefolds (that can be deformed to
Moishezon manifolds) are parametrized by a single universal moduli space in which families of
smooth Calabi-Yau threefolds of different homotopy types are connected by conifold transitions. In order to develop a metric approach to \emph{Reid's fantasy}, Yau has proposed to study special types of balanced metrics which endow the Calabi-Yau threefolds with a preferred geometry \cite{FuYau,LiYau}. Note here that, similarly as  in K\"ahler geometry, balanced metrics arise in infinite-dimensional families, each of them parametrized by a class in the \emph{balanced cone} in Bott-Chern cohomology \cite{FuXiao} (see Remark \ref{rem:balanceddef}). As we will see in Section \ref{sec:Strominger}, a natural way to rigidify balanced metrics is imposing conditions on the torsion $3$-form of the Bismut connection $d^c \omega$. An alternative (and in some sense orthogonal) approach in the literature is to fix the volume form $\omega^n/n!$ of the balanced metric \cite{SzToWe}. In particular, this gives special solutions of the dilatino equation, where the left and right hand side of \eqref{eq:dilatinosec2} vanish independently. 

\section{Hermite-Einstein metrics on balanced manifolds}\label{sec:HermEin}

\subsection{The Hermite-Einstein equation and stability}\label{sec:HK}

Let $X$ be a compact complex manifold of dimension $n$ endowed with a balanced metric $g$. Let $\cE$ be a holomorphic vector bundle over $X$ of rank $r$, with underlying smooth complex vector bundle $E$. We will denote by $\Omega^{p,q}(E)$ the space of $E$-valued $(p,q)$-forms on $X$.  The holomorphic structure on $E$ given by $\cE$ is equivalent to a Dolbeault operator
$$
\dbar_\cE \colon \Omega^0(E) \to \Omega^{0,1}(E)
$$
satisfying the integrability condition $\dbar_\cE^2 = 0$ (see e.g. \cite{Gauduchon}). 

Given a hermitian metric $h$ on $E$, there is an associated unitary \emph{Chern connection} $A$ compatible with the holomorphic structure $\cE$, uniquely defined by the properties
\begin{align*}
d_A^{0,1} & = \dbar_\cE,\\
d(s,t)_h & = (d_A s,t)_h + (s,d_A t)_h,
\end{align*}
for $s,t \in \Omega^0(E)$. Here, $d_A \colon \Omega^0(E) \to \Omega^1(E)$ is the covariant derivative defined by $A$. In a local holomorphic frame $\{e_j\}_{j=1}^r$, the Chern connection is given by the matrix-valued $(1,0)$-form
$$
h^{-1}\partial h,
$$
where $h_{ij} = (e_j,e_i)_h$. The curvature of the Chern connection is a $(1,1)$-form with values in the skew-hermitian endomorphisms $\End (E,h)$ of $E$, defined by
$$
F_h = d_A^2 \in \Omega^{1,1}(\End (E,h)).
$$
In a local holomorphic frame, we have the formula
\begin{equation}\label{eq:curvaturelocal}
F_h = \dbar (h^{-1}\partial h).
\end{equation}

\begin{definition}\label{def:HE}
The Hermite-Einstein equation, for a hermitian metric $h$ on $E$, is
\begin{equation}\label{eq:HE}
i \Lambda_\omega F_h = \lambda \Id.
\end{equation}
\end{definition}

In equation \eqref{eq:HE}, $\Lambda_\omega$ is the contraction operator \eqref{eq:Lambda}, $\Id$ denotes the identity endomorphism on $E$, and $\lambda \in \RR$ is a real constant. The Hermite-Einstein equation is a non-linear second-order partial differential equation for the hermitian metric $h$, as it follows from \eqref{eq:curvaturelocal}.

To understand the existence problem for the Hermite-Einstein equation, the first basic observation is that, in order for a solution to exist, the constant $\lambda \in \RR$ must take a specific value fixed by the topology of $(X,g)$ and $E$, in the following sense. Consider the balanced class of $\omega^{n-1}$ in de Rham cohomology
\begin{equation}\label{eq:balancedclassdR}
\tau := [\omega^{n-1}] \in H^{2n-2}(X,\RR),
\end{equation}
determined by the balanced metric $g$. Recall that the first Chern class $c_1(E)$ of $E$ is represented by
$$
c_1(E) = [i\tr F_h/2\pi] \in H^2(X,\ZZ),
$$
for any choice of hermitian metric $h$.

\begin{definition}\label{def:degree}
The $\tau$-degree of $E$ is
\begin{equation}\label{eq:slope}
\deg_\tau(E) = c_1(E) \cdot \tau,
\end{equation}
where $c_1(E) \cdot \tau \in H^{2n}(X,\RR) \cong \RR$ denotes the cup product in cohomology.
\end{definition}
Taking the trace in \eqref{eq:HE} and integrating against $ \frac{\omega^n}{n!} $ we obtain
\begin{equation}\label{eq:lambdavalue}
\lambda = \frac{2\pi}{(n-1)!}\frac{\deg_\tau(E)}{r \Vol_\omega},
\end{equation}
where $\Vol_\omega = \int_X \frac{\omega^n}{n!}$ is the volume. When $E$ is a line bundle, the Hermite-Einstein equation can always be solved for this value of $\lambda$.

\begin{proposition}
For a holomorphic line bundle $\cE$ the Hermite-Einstein equation \eqref{eq:HE} admits a unique solution $h$ up to homothety, provided that $\lambda$ is given by \eqref{eq:lambdavalue}.
\end{proposition}
\begin{proof}
Fixing a reference hermitian metric $h_0$ any other metric $h$ on the line bundle is given by $h = e^f h_0$ and, using this, equation \eqref{eq:HE} is equivalent to 
\begin{equation}\label{eq:iddbar}
i \Lambda_\omega \dbar \partial f = \lambda - i \Lambda_\omega F_h.
\end{equation}
By \cite[eq. (25)]{Gau1984}, the balanced condition can be alternatively written as an equality of differential operators on smooth functions on $X$
$$
2 i \Lambda_\omega \dbar \partial = \Delta_\omega := dd^* + d^*d,
$$
and therefore $i \Lambda_\omega \dbar \partial$ is self-adjoint, elliptic, with Kernel given by $\RR \subset C^\infty(X)$. By \eqref{eq:lambdavalue}, 
$$
\int_X (\lambda - i \Lambda_\omega F_h) \frac{\omega^n}{n!} = 0,
$$
so $\lambda - i \Lambda_\omega F_h$  is orthogonal to $\RR$ in $C^\infty(X)$. We conclude that \eqref{eq:iddbar} has a unique smooth solution $f$ with $\int_X f \frac{\omega^n}{n!} = 0$.
\end{proof}

For higher rank bundles, the existence of solutions of the Hermite-Einstein equation relates to an algebraic numerical condition for $\cE$ -- originally related to the theory of quotients of algebraic varieties by complex reductive Lie groups, known as Geometric Invariant Theory \cite{MFK}. To state the precise result, we need to extend Definition \ref{def:degree} to arbitrary torsion-free coherent sheaves of $\mathcal{O}_X$-modules (for the basic definitions we refer to \cite[Section 5-6]{Kob}). Given such a sheaf $\cF$ of rank $r_\cF$, the determinant of $\cF$, defined by 
$$
\det \cF := (\Lambda^{r_\cF} \cF)^{**}
$$
is a holomorphic line bundle 
(such that $\cF = \det \cF$ when $\cF$ is torsion-free of rank $1$), and we can extend Definition \ref{def:degree} setting
$$
\deg_\tau (\cF) := \deg_\tau (\det \cF).
$$
We define the $\tau$-slope of a torsion-free coherent sheaf $\cF$ of $\mathcal{O}_X$-modules by
\begin{equation}\label{eq:slopecoh}
\mu_\tau(\cF) = \frac{\deg_\tau(\cF)}{r_\cF}.
\end{equation}

\begin{definition}\label{def:stable}
A torsion-free sheaf $\cF$ over $X$ is 
\begin{enumerate}[i)]
\item $\tau$-(semi)\emph{stable} if for every subsheaf $\cF' \subset \cF$ with $0 < r_{\cF'} < r_{\cF}$ one has 
$$\mu_\tau(\cF') < (\leqslant) \mu_\tau(\cF),$$
\item  $\tau$-\emph{polystable} if $\cF = \bigoplus_j \cF_j$ with $\cF_j$ stable and $\mu_\tau(\cF_i) = \mu_\tau(\cF_j)$ for all $i,j$.
\end{enumerate}
\end{definition}

When $X$ is projective and $g$ is a K\"ahler Hodge metric, with K\"ahler class associated to a hyperplane section of $X$, this definition coincides with the original definition of slope stability due to Mumford and Takemoto (see e.g. \cite{MFK}).

We can state now the characterization of the existence of solutions for the Hermite-Einsten equation \eqref{eq:HE}.

\begin{theorem}\label{th:LiYau}
There exists a Hermite-Einstein metric $h$ on $\cE$ if and only if $\cE$ is $\tau$-polystable.
\end{theorem}

This result was first proved by Narasimhan and Seshadri in the case of curves \cite{NS}. The `only if part' was proved in higher dimensions by Kobayashi \cite{Kob1982}. The `if part' was proved for algebraic surfaces by Donaldson \cite{D3}, and for higher dimensional compact K\"ahler manifolds by Uhlenbeck and Yau \cite{UY}. Buchdahl extended Donaldson's result to arbitrary compact complex surfaces in \cite{Buchdahl}, and Li and Yau generalized Uhlenbeck and Yau's theorem to any compact complex hermitian manifold in \cite{LiYauHYM}. 

\begin{remark}\label{rem:degree}
The statement of Li-Yau Theorem \cite{LiYauHYM} for an arbitrary hermitian metric $g$ uses that the existence of a Hermite-Einstein metric only depends on the conformal class of $g$ (see \cite[Lem. 2.1.3]{lt}), and therefore $g$ can be assumed to be Gauduchon. As for the stability condition, when $g$ is a Gauduchon metric one defines the $g$-slope of $\cE$ by $\mu_g(\cE) = \deg_g(\cE)/r$, where
$$
\deg_g(\cE) =  \frac{i}{2\pi}\int_X  \tr F_h \wedge \omega^{n-1}
$$
for a choice of hermitian metric $h$. The $g$-degree $\deg_g(\cE)$ is independent of $h$ by Stokes Theorem, since $dd^c \omega^{n-1} = 0$ and $\tr F_h = \tr F_{h_0} + \dbar \partial f$ for $h = e^f h_0$. 
In general, $\deg_g$ is not topological, and depends on the holomorphic structure of $\det \cE$ (e.g. for $n=2$ the $g$-degree is topological if and only if $X$ is k\"ahlerian \cite[Cor. 1.3.14]{lt}).
\end{remark}

On general grounds, an effective check of any of the two equivalent conditions in Theorem \ref{th:LiYau} is a difficult problem. We comment on a class of examples, given by deformations of (essentially) tangent bundles of algebraic Calabi-Yau manifolds. We postpone the examples of solutions of the Hermite-Einstein equations on non-k\"ahlerian balanced manifolds to Section \ref{sec:existence}.

\begin{example}\label{ex:CYHE}
Given a K\"ahler-Einstein metric $g$ on $X$, it is easy to check that the induced Chern connection on $TX$ is Hermite-Einstein. In particular, by Yau's solution of the Calabi Conjecture \cite{Yau1977}, the tangent bundle of any k\"ahlerian Calabi-Yau manifold is polystable with respect to any K\"ahler class, and it is stable if $b_1(X) = 0$. Since being stable is an open condition, any deformation of the tangent bundle of a simply connected Calabi-Yau manifold is stable. As concrete examples (see e.g. \cite{Huyb}): in dimension $2$, the tangent bundle of a $K3$ surface has unobstructed deformations and, in dimension $3$, the tangent bundle of a simply connected complete intersection Calabi-Yau threefold has unobstructed deformations. For a quintic hypersurface $X \subset \CC\PP^4$, the space of deformations of $TX$ is $224$-dimensional. For $X$ the generic quintic (see Example \ref{ex:Candelas}), Huybrechts has proved that the polystable bundle $TX \oplus \mathcal{O}_X$ admits stable holomorphic deformations \cite{Huyb}, with the following property: they have non-trivial restriction to any rational curve of degree one.
\end{example}

It is interesting to compare the previous example with a result by Chuan \cite{Chuan}. Let $X \to X_0 \dashrightarrow X_t$ be a conifold transition between Calabi-Yau threefolds, as in Section \ref{sec:conifoldtr}. Let $\cE$ be a stable holomorphic vector bundle over $X$ with $c_1(\cE) = 0$, and assume that $\cE$ is trivial in a neighbourhood of the exceptional rational curves. Assume further the exists a stable bundle $\cE_t$ over $X_t$, given by a deformation of the push-forward of $\cE$ along $X \to X_0$. Then, Chuan has proved under this hypothesis that $\cE_t$ is stable with respect to the balanced metric constructed by Fu, Li and Yau \cite{FuLiYau}.

\subsection{Gauge theory, K\"ahler reduction and the necessity of stability}\label{sec:gauge}

In this section we give a (lengthy) geometric proof of the `only if part' of Theorem \ref{th:LiYau}, following \cite{MR}. This method of proof is based on the correspondence between symplectic quotients and GIT quotients, given by the Kemp-Ness Theorem \cite{KN} and uses some of the ingredients required for the `if part' of Theorem \ref{th:LiYau}. We note that the standard proof of the `only if part' 
(see e.g. \cite{lt}) -- which turns on the principle that ``curvature decreases in holomorphic sub-bundles and increases in holomorphic quotients'' \cite{AB} -- is shorter, and does not use any gauge-theoretical methods. Nonetheless, we find the proof of \cite{MR} more pedagogical, and better suited for the purposes of these notes.
By the end, we discuss briefly the implications of this method for the geometry of the moduli space of stable vector bundles.

Let $E$ be a smooth complex vector bundle of rank $r$ over $X$. Let $\cG^c$ be the gauge group of $E$, that is, the group of diffeomorphisms of $E$ projecting to the identity on $X$ and $\CC$-linear on the fibres. Consider the space $\cC$ of Dolbeault operators
$$
\dbar_\cE \colon \Omega^0(E) \to \Omega^{0,1}(E)
$$
on $E$, which is a complex affine space modelled on $\Omega^{0,1}(\End E)$. Then, $\cG^c$ acts on $\cC$ by
$$
g \cdot \dbar_\cE = g \dbar_\cE g^{-1},
$$
and preserves the (constant) complex structure on $\cC$.

Fix now a hermitian metric $h$ on $E$. Consider the space $\cA$ of unitary connections on $(E,h)$, which is an affine space modelled on $\Omega^1(\End (E,h))$. The \emph{unitary gauge group} $\cG \subset \cG^c$, given by automorphisms of $E$ preserving $h$, acts on $\cA$ by
$$
g \cdot d_A = g d_A g^{-1},
$$
preserving the symplectic structure
\begin{equation}\label{eq:SymfC}
\omega_{\cA}(a,b) = - \int_X \tr(a \wedge b) \wedge \omega^{n-1}
\end{equation}
for $a,b \in T_A \cA = \Omega^1(\End (E,h))$, with $A\in\cA$. 

There is a real affine bijection between the two infinite-dimensional spaces $\cA$ and $\cC$, defined by 
\begin{equation}\label{eq:affine}
\cA \to \cC \colon A \to \dbar_A = (d_A)^{0,1},
\end{equation}
with inverse given by the Chern connection of $h$ in $\cE = (E,\dbar_\cE)$. Under this bijection, the integrability condition $\dbar_A^2 = 0$ is equivalent to
\begin{equation}\label{eq:FA02}
F_A^{0,2} = 0,
\end{equation}
and the complex structure on $\cC$ translates to $a \to Ja$, for $J$ the almost complex structure on $X$. The symplectic form \eqref{eq:SymfC} is compatible with this complex structure, and induces a K\"ahler structure on $\cC$.

Using the bijection \eqref{eq:affine} and the K\"ahler form \eqref{eq:SymfC}, we can now give a geometric interpretation to the Hermite-Einstein equation \eqref{eq:HE}. The following observation is due to Atiyah and Bott \cite{AB} when $X$ is a Riemann surface, and was generalized by Donaldson \cite{D3} to higher dimensional K\"ahler manifolds. As pointed out by L\"ubke and Teleman \cite[Sec. 5.3]{lt}, remarkably the construction only needs that $g$ is balanced.

\begin{proposition}\label{prop:mmap}
The $\cG$-action on $\cA$ is Hamiltonian, with equivariant moment map $\mu\colon \cA\to
(\LieG)^*$ given by
\begin{equation}
\label{eq:momentmap-cG}
  \langle\mu(A),\zeta\rangle = - \int_X \tr \zeta  (\Lambda_\omega F_A + i\lambda \Id) \frac{\omega^n}{n},
\end{equation}
where $\zeta \in \Omega^0(\End (E,h)) \cong \Lie \cG$.
\end{proposition}

\begin{proof}
The $\cG$-equivariance follows from $F_{g \cdot A} = g F_A g^{-1}$ for any $g \in \cG$. Thus, given $a \in \Omega^1(\End (E,h))$ we need to prove that
$$
\langle d\mu(a),\zeta\rangle = \omega_\cA(Y_\zeta,a)
$$
where $Y_\zeta$ denotes the infinitesimal action of $\zeta$, given by
$$
Y_\zeta(A) = - d_A \zeta.
$$
Using $\delta_a F_A = d_A a$ and $d \omega^{n-1} = 0$, integration by parts gives
\begin{align*}
\langle d\mu(a),\zeta\rangle & = - \int_X \tr \zeta  d_A a \wedge \omega^{n-1}\\
& = \int_X \tr d_A \zeta \wedge  a \wedge \omega^{n-1} = \omega_\cA(Y_\zeta,a).
\end{align*}
\end{proof}

\begin{definition}\label{def:HYM}
A unitary connection $A \in \cA$ is called a \emph{Hermite-Yang-Mills connection} if it satisfies $A \in \mu^{-1}(0)$ and $\dbar_A^2 = 0$, that is,
$$
i \Lambda_\omega F_A = \lambda \Id, \qquad \qquad F_A^{0,2} = 0.
$$
\end{definition}

For a Hermite-Yang-Mills connection $A$, the metric $h$ on $\cE = (E,\dbar_A)$ is Hermite-Einstein. Using this fact, we can now prove the following.

\begin{theorem}\label{th:necessity}
If there exists a Hermite-Einstein metric $h$ on $\cE$, then $\cE$ is $\tau$-polystable.
\end{theorem}

\begin{proof}
Let $h$ be a Hermite-Einstein metric on $\cE$ and $\cF \subset \cE$ a coherent subsheaf with $0 < r_\cF < r$. We can assume that $\cF$ is reflexive \cite[Prop. 1.4.5]{lt}. Using a characterization of reflexive sheaves in terms \emph{weakly holomorphic subbundles} \cite{Simpson,UY} (see also \cite[p. 81]{lt}), there exists an analytic subset $S \subset X$ of codimension $\geqslant 2$ and $\pi \in L^2_1(\End E)$ such that
\begin{equation}\label{eq:conditionspi}
\pi^* = \pi = \pi^2, \qquad (\Id - \pi) \dbar_\cE \pi = 0
\end{equation}
on $L^1(\End E)$, $\pi_{|X \backslash S}$ is smooth and satisfies \eqref{eq:conditionspi}, and 
$$
\cF' = \cF_{|X \backslash S} = \operatorname{Im} \pi_{|X \backslash S}
$$ 
is a holomorphic subbundle of $\cE' = \cE_{|X \backslash S}$ . Furthermore, the curvature of $h_{|X \backslash S}$ on $\cF'$ defines a closed current on $X$ which represents $-i2\pi c_1(\cF) \in H^2(X,\CC)$. 

Using the orthogonal projection $\pi$, we can define a \emph{weak element} in the Lie algebra of $\cG$ by
$$
\zeta = i(\pi - \nu(\Id - \pi)).
$$
Associated to $\zeta$, there is a $1$-parameter family of (singular) Dolbeault operators 
$$
\dbar_{\cE_t} = e^{it\zeta}\cdot \dbar_\cE,
$$
and, since $\CC^*\Id \subset \cG^c$ acts trivially on $\cC$, we can assume the normalization
$$
\nu = \frac{r_\cF}{r - r_\cF}.
$$
We want to calculate the \emph{maximal weight}
\begin{equation}\label{eq:weightineq}
w(\dbar_\cE,\zeta):= \lim_{t \to +\infty} \langle \mu(\dbar_{\cE_t}),\zeta\rangle = \lim_{t \to +\infty} - \int_X \tr (\zeta F_{h,\dbar_{\cE_t}})\omega^n/n.
\end{equation}
Note that we can do the calculation in the smooth locus of $\pi$, since $S$ is codimension $2$. Away from the singularities of $\pi$, we have
$$
\dbar_{\cE_t} = \pi \dbar_\cE \pi + (\Id - \pi) \dbar_\cE (\Id - \pi) + e^{-t(1+\nu)}\pi \dbar_\cE (\Id - \pi)
$$
and therefore
$$
\dbar_{E_\infty} := \lim_{t \to +\infty} \dbar_{\cE_t} = \pi \dbar_\cE \pi + (\Id - \pi) \dbar_\cE (\Id - \pi),
$$
which corresponds to the direct sum
$$
(E, \dbar_{E_\infty})_{|X \backslash S} \cong \cF' \oplus \cE' / \cF'.
$$
Using this, we have
\begin{equation}\label{eq:weight}
\begin{split}
w(\dbar_\cE,\zeta) &= - \int_X  (i \tr \Lambda_\omega F_{h,\dbar_{\cF'}})\omega^n/n + \nu \int_X  (i \tr \Lambda_\omega F_{h,\dbar_{\cE'/\cF'}})\omega^n/n\\
& = 2\pi (n-1)! \(- \deg_\tau (\cF) + \frac{r_\cF}{r - r_\cF} (\deg_\tau (\cE) - \deg_\tau (\cF))\)\\
& = 2\pi (n-1)! \frac{r r_\cF}{r - r_\cF}(\mu_\tau(\cE) - \mu_\tau(\cF)).
\end{split}
\end{equation}
The key point of the proof is the monotonicity of $\langle \mu(\dbar_{\cE_t}),\zeta \rangle$, which follows from the positivity of the K\"ahler form \eqref{eq:SymfC},
$$
\frac{d}{dt} \langle \mu(\dbar_{\cE_t}),\zeta \rangle = |Y_{\zeta|\dbar_{\cE_t}}|^2 = (1+\nu)^2e^{-2t(1+\nu)}\|\pi \dbar_\cE (\Id - \pi)\|_{L_2}^2
$$
combined with the Hermite-Einstein condition, which gives  $\langle \mu(\dbar_\cE),\zeta \rangle = 0$. Combining these two facts,
\begin{equation}\label{eq:weightineq2}
w(\dbar_\cE),\zeta) =  \int_0^\infty |Y_{\zeta|e^{it\zeta}u_{\underline s}}|^2dt = \frac{1}{2}(1 + \nu) \|\pi \dbar_\cE (\Id - \pi)\|_{L_2}^2 \geq 0,
\end{equation}
and therefore $\mu_\tau(\cE) \geqslant \mu_\tau(\cF)$ by \eqref{eq:weight}. We conclude that $\cE$ is semistable. 

Suppose now that $\cE$ is not stable and that we have an equality in \eqref{eq:weightineq2}. Then, $\pi \dbar_\cE (\Id - \pi) = 0$ and
$$
\dbar_\cE(\pi) := \dbar_\cE\pi - \pi \dbar_\cE = \dbar_\cE \pi - \pi \dbar_\cE (\pi + (\Id - \pi)) = 0,
$$
and therefore $\pi^*$ is in the Kernel of the elliptic operator $i \Lambda_\omega \dbar_\cE \partial_\cE$ (see \cite[Lem. 7.2.3]{lt}) and hence it is smooth on $X$. We conclude that $\cF$ and $\cE / \cF$ are holomorphic vector bundles, and we have an orthogonal decomposition
$$
\cE \cong \cF \oplus \cE / \cF
$$
with $\mu_\tau(\cE) = \mu_\tau(\cF) = \mu_\tau(\cE/\cF)$. In addition, $\cF$ and $\cE / \cF$ inherit Hermite-Einstein metrics by restriction. Induction on the rank of $\cE$ completes the argument.
\end{proof}

To conclude, we discuss briefly the implications of the existence of the previous infinite-dimensional K\"ahler structure, for the geometry of the moduli space of stable vector bundles and Hermite-Yang-Mills connections. Let
$$
\cC^s \subset \cC
$$
be the subset of integrable Dolbeault operators $\dbar_\cE$ which define $\tau$-stable holomorphic vector bundles $\cE$. Two integrable Dolbeault operators are in the same $\cG^c$-orbit if and only if define isomorphic vector bundles, and therefore the $\cG^c$-action preserves $\cC^s$. We define the \emph{moduli space of $\tau$-stable vector bundles} by the quotient
$$
\cC^s/\cG^c.
$$ 
This quotient has a natural Hausdorff topology, and can be endowed with a finite dimensional complex analytic structure (which may be non-reduced) \cite[Cor. 4.4.4]{lt}.

Let $\cA^* \subset \cA$ be the $\cG$-invariant subset of irreducible connections which satisfy the integrability condition \eqref{eq:FA02}. Recall that a connection $A$ is irreducible if the Kernel of the induced covariant derivative $d_A$ in $\End (E,h)$ equals $i \RR \Id$. For $A \in \mu^{-1}(0) \cap \cA^*$, $h$ is a Hermite-Einstein metric on $\cE = (E,\dbar_A)$. We define the \emph{moduli space of irreducible Hermite-Yang-Mills connections} as the quotient
$$
\mu^{-1}(0) \cap \cA^*/\cG.
$$ 
This set has a natural Hausdorff topology, and can be endowed with a finite dimensional real analytic structure (which may be non-reduced) \cite[Prop. 4.2.7]{lt}. Furthermore, by Proposition \ref{prop:mmap} the moduli space inherits a real-analytic symplectic structure away from its singularities \cite[Cor. 5.3.9]{lt}.

Building on the proof of Theorem \ref{th:necessity}, one can prove that for $\mu^{-1}(0) \cap \cA^*$ one has $\dbar_A \in \cC^s$ (see \cite[Remark 2.3.3]{lt}) and therefore there is a natural map
$$
\mu^{-1}(0) \cap \cA^*/\cG \to \cC^s/\cG^c.
$$
This map induces a real analytic isomorphism by Theorem \eqref{th:LiYau}, and a K\"ahler structure on $\cC^s/\cG^c$ away from its singularities (see \cite[Cor. 4.4.4]{lt} and \cite[Cor. 5.3.9]{lt}).

\section{The Strominger system}\label{sec:Strominger}

\subsection{Definition and first examples}\label{sec:defi}

Let $(X,\Omega)$ be a Calabi-Yau manifold of dimension $n$, with underlying smooth manifold $M$ and almost complex structure $J$. Let $\cE$ be a holomorphic vector bundle over $X$, with underlying smooth complex vector bundle $E$. To define the Strominger system, we consider integrable Dolbeault operators $\dbar_T$, that is, satisfying $\dbar_T^2 = 0$, on the smooth complex vector bundle $(TM,J)$.

\begin{definition}\label{def:Stromingersystem}
The \emph{Strominger system}, for a hermitian metric $g$ on $(X,\Omega)$, a hermitian metric $h$ on $\cE$, and an integrable Dolbeault operator $\dbar_T$ on $(TM,J)$, is
\begin{equation}\label{eq:Stromingersystem}
\begin{split}
\Lambda_\omega F_h & = 0,\\
\Lambda_\omega R & = 0,\\
d^*\omega - d^c \log \|\Omega\|_\omega & = 0,\\
dd^c \omega - \alpha \(\tr R \wedge R - \tr F_h \wedge F_h\) & = 0. 
\end{split}
\end{equation} 
\end{definition}

Here, $\alpha$ is non-vanishing real constant and $R = R_{g,\dbar_T}$ denotes the curvature of the Chern connection of $g$, regarded as a hermitian metric on the holomorphic vector bundle $\cT = (TM,J,\dbar_T)$. The first two equations in the system correspond to the Hermite-Einstein condition for the curvatures $F_h$ and $R$ with vanishing constant $\lambda$, with respect to the K\"ahler form $\omega$ (see Definition \ref{def:HE}). The third equation, involving the K\"ahler form $\omega$ and the holomorphic volume form $\Omega$, is the dilatino equation \eqref{eq:dilatinosec2}. The new ingredient in the system is an equation for $4$-forms, known as the \emph{Bianchi identity}
\begin{equation}\label{eq:bianchi}
dd^c \omega - \alpha \(\tr R \wedge R - \tr F_h \wedge F_h\) = 0,
\end{equation}
which intertwines the exterior differential of the torsion $- d^c\omega$ of the Bismut connection of $g$, with the curvatures $F_h$ and $R$. 

\begin{remark}
The second equation in \eqref{eq:Stromingersystem}, that is, $\Lambda_\omega R = 0$, is often neglected in the literature. In this case, $R$ is typically taken to be the Chern connection of $g$ on the holomorphic tangent bundle $TX$ (see e.g. \cite{LiYau}). Motivation for considering the Hermite-Einstein condition for $R$ comes from physics, and it will be explained in Section \ref{sec:sugra}.
\end{remark}

Fixing the holomorphic structure $\dbar_T$ on the smooth complex vector bundle $(TM,J)$ essentially determines the hermitian metric which solves the Hermite-Einstein equation on $(TM,J,\dbar_T)$ (see Section \ref{sec:gauge}), and therefore lead us to an overdetermined system of equations. To put the bundles $E$ and $(TM,J)$ on equal footing, it is convenient to take a gauge-theoretical point of view, by fixing the hermitian metric $h$ on $E$ and substitute the unknowns $h$ and $\dbar_T$ in Definition \ref{def:Stromingersystem} by a unitary connection $A$ on $(E,h)$ and a unitary connection $\nabla$ on $(TM,J,g)$, with curvature $R_\nabla$ (cf. Section \ref{sec:gauge}). This lead us to the following equivalent definition of the Strominger system.

\begin{definition}\label{def:Stromingersystem2}
The \emph{Strominger system}, for a hermitian metric $g$ on $(X,\Omega)$, a unitary connection $A$ on $(E,h)$ and a unitary connection $\nabla$ on $(TM,J,g)$, is
\begin{equation}\label{eq:Stromingersystem2}
\begin{split}
\Lambda_\omega F_A & = 0, \qquad F_A^{0,2} = 0 \\
\Lambda_\omega R_\nabla & = 0, \qquad R_\nabla^{0,2} = 0 \\
d( \|\Omega\|_\omega \omega^{n-1}) & = 0,\\
dd^c \omega - \alpha \(\tr R_\nabla \wedge R_\nabla - \tr F_A \wedge F_A\) & = 0. 
\end{split}
\end{equation} 
\end{definition}

From this alternative point of view, the Strominger system couples a pair of Hermite-Yang-Mills connections $A$ and $\nabla$ (see Definition \ref{def:HYM}) with a conformally balanced metric $\omega$, by means of the Bianchi identity \eqref{eq:bianchi}. For the equivalence between \eqref{eq:Stromingersystem} and \eqref{eq:Stromingersystem2}, we use Li-Yau characterization of the dilatino equation \eqref{eq:dilatinosec2} in terms of the conformally balanced equation \eqref{eq:LiYau}.

The system \eqref{eq:Stromingersystem2} makes more transparent three types of necessary conditions for the existence of solutions of the Strominger system. Firstly, for $(X,\Omega,E)$ to admit a solution of the Strominger system there are some evident cohomological obstructions on the Chern classes
\begin{equation}\label{eq:c1}
\deg_\tau(E) = 0, \qquad  c_1(X) = 0,
\end{equation}
and also
\begin{equation}\label{eq:c2BC}
ch_2(E) = ch_2(X) \in H^{2,2}_{BC}(X,\RR),
\end{equation}
where $ch_2(E)$ and $ch_2(X)$ denote the second Chern character of $E$ and $X$. On a general complex manifold, the condition \eqref{eq:c2BC} depends on the complex structure of $X$, since the natural map
$$
H^{2,2}_{BC}(X,\RR) \to H^4(X,\RR)
$$
may have a Kernel. Secondly, we have two further conditions on the complex structure on $X$, that is, the complex manifold must have trivial canonical bundle and it must be balanced. Recall that the balanced condition for $X$ can be expressed as a positivity condition on the homology of $X$, which involves complex currents (see Proposition \ref{prop:cohomol} and \cite{Michel}). Finally, if $(g,\nabla,A)$ is a solution with balanced class
$$
\tau = [\|\Omega\|_\omega \omega^{n-1}] \in H^{n-1,n-1}_{BC}(X,\RR),
$$
Theorem \ref{th:LiYau} implies that the holomorphic bundle $\cE = (E,\dbar_A)$ and the holomorphic bundle $\cT = (TM,J,\dbar_\nabla)$ must be $\tau$-polystable.

In the rest of this section we discuss some basic examples of solutions of the Strominger system. We postpone more complicated existence results to Section \ref{sec:existence}. Let us start with complex dimension $1$. 

\begin{example}\label{ex:curve}
When $n=1$, $X$ is forced to be an elliptic curve and the Bianchi identity is an empty condition. Furthermore, the dilatino equation reduces to the K\"ahler Ricci-flat condition on $X$ (see Section \ref{sec:dilatino}). Hence, by Theorem \ref{th:LiYau} solutions of the Strominger system with a prescribed K\"ahler class are given by degree-zero polystable holomorphic vector bundles over the elliptic curve $X$ (see e.g. \cite{Tu}).
\end{example}

The next example shows that the system can always be solved for Calabi-Yau surfaces, provided that \eqref{eq:c1} and \eqref{eq:c2BC} are satisfied. Note that a compact Calabi-Yau surface must be a $K3$ surface or a complex torus, and therefore is always k\"ahlerian. Hence, the condition \eqref{eq:c2BC} in this case is topological. We follow an argument of Strominger in \cite{Strom} and do not assume that the connection $\nabla$ is unitary (to the knowledge of the author, with the unitary assumption there is no general existence result for $n=2$).

\begin{example}\label{ex:surface}
Let $(X,\Omega)$ be a compact Calabi-Yau manifold of dimension $2$. Let $g$ be a K\"ahler Ricci-flat metric on $X$, with K\"ahler form $\omega$. By Example \ref{ex:CY}, $g$ solves the dilatino equation. Let $\nabla$ be a unitary Hermite-Yang-Mills connection on the smooth hermitian bundle $(TM,J,g)$ (e.g. we can take the Chern connection of $g$). Let $(E,h)$ be a smooth complex hermitian vector bundle satisfying \eqref{eq:c1} and \eqref{eq:c2BC}, with a Hermite-Yang-Mills connections $A$. We want to show that we can find a metric $\tilde g$ on the conformal class of $g$ solving the Strominger system. Consider $\tilde g = e^f g$ for $f \in C^\infty(X)$. Using that $n = 2$ we have
$$
\|\Omega\|_{\tilde \omega} \tilde \omega^{n-1} = \|\Omega\|_\omega \omega^{n-1},
$$
and therefore $\tilde \omega$ solves the dilatino equation. Furthermore, $\nabla$ and $A$ are Hermite-Yang-Mills connections also for $\tilde \omega$ (note that $\nabla$ is no longer $\tilde \omega$-unitary). Hence, to solve the Strominger system \eqref{eq:Stromingersystem2} with this ansatz, we just need to solve the Bianchi identity \eqref{eq:bianchi} for $\tilde g$. Now, using that $g$ is K\"ahler, \eqref{eq:bianchi} is equivalent to
$$
\Delta_\omega (e^f) = \Lambda_\omega^2 \alpha (\tr R_\nabla \wedge R_\nabla - \tr F_A \wedge F_A).
$$
The obstruction to solve this equation is $ch_2(E) = ch_2(X)$ in $H^{4}(X,\RR)$, which holds by assumption. Note that we can assume our solution to be positive, since $X$ is compact. As a concrete example, we can consider $X$ to be a $K3$ surface and take $\cT = TX$ and $\cE$ a small holomorphic deformation of $TX$ or $TX \oplus \mathcal{O}_X$ \cite{Huyb}, and apply Theorem \ref{th:LiYau}.
\end{example}

We discuss next the arguably most basic examples of solutions of the Strominger system in dimension $n \geq 3$.

\begin{example}\label{ex:standard}
Let $(X,\Omega)$ be a compact k\"ahlerian Calabi-Yau manifold of dimension $n$. Let $g$ be a K\"ahler Ricci-flat metric on $X$. By Example \ref{ex:CY}, $g$ solves the dilatino equation and by Example \ref{ex:CYHE}, the Levi-Civita connection $\nabla^g$ is Hermite-Einstein. Set $\nabla = \nabla^g$ and denote by $h_0$ a constant hermitian metric on the trivial bundle $\oplus_{i=1}^{r-n} \mathcal{O}_X$ over $X$. Define
$$
\cE = TX \bigoplus (\oplus_{i=1}^{r-n} \mathcal{O}_X),
$$
that we consider endowed with the hermitian metric $g \oplus h_0$, with Chern connection $A = \nabla \oplus d$. Then, it is immediate that $g \oplus h_0$ is Hermite-Einstein. Finally, the Bianchi identity \eqref{eq:bianchi} is satisfied, because $tr R^2 = tr F_A^2$ and $dd^c \omega = 0$, since $g$ is K\"ahler.
\end{example}

When $n = r = 3$, these are called \emph{standard embedding solutions} in the physics literature, based on the natural monomorphism of Lie algebras
$$
\mathfrak{su}(3) \rightarrow \mathfrak{e}_8,
$$
where $\mathfrak{e}_8$ is the Lie algebra of the compact exceptional group $E_8$.

\subsection{Existence results}\label{sec:existence}

The Strominger system is a fully non-linear coupled system of partial differential equations for the hermitian metric $g$, and the unitary connections $\nabla$ and $A$. Note that the system is of mixed order, since the Hermite-Yang-Mills equations are of first order in $\nabla$ and $A$, the dilatino equation is of order one in the K\"ahler form $\omega$, while the Bianchi identity is of second order in $\omega$. The most demanding and less understood condition of the Strominger system is, indeed, the Bianchi identity \eqref{eq:bianchi}. In dimension $n \geqslant 3$, this condition is the ultimate responsible of the non-K\"ahler nature of this PDE problem, as the non-vanishing of the \emph{Pontryagin term} $\operatorname{tr} R_\nabla^2 - \operatorname{tr} F_A^2$ prevents the hermitian form $\omega$ to be closed and hence allows the complex manifold $X$ to be non-k\"ahlerian. To the present day, we have a very poor understanding of the Bianchi identity from an analytical point of view.

The first solutions of the system for $\cE$ a stable holomorphic vector bundle with rank $r = 4,5$ over an algebraic Calabi-Yau threefold were found by Li and Yau \cite{LiYau} (cf. Example \ref{ex:standard}). Solutions in non-k\"ahlerian threefolds were first obtained by Fu and Yau \cite{FuYau}, on suitable torus fibrations over a $K3$ surface. Further solutions in non-k\"ahlerian homogeneous spaces, specially on nilmanifolds, have been found over the last years (see \cite{FIVU,FeiYau,Grant,OUV} and references therein). For examples in non-compact threefolds we refer to \cite{FTY,Fei1,Fei2,FIUVa}.

There are essentially three known methods to solve the Strominger system in a compact complex threefold: by perturbation in k\"ahlerian manifolds, by reduction in a non-k\"ahlerian fibration over a K\"ahler manifold, and the method of invariant solutions in homogeneous spaces. The aim of this section is to illustrate these three methods with concrete results and examples. By the end, we will comment on a conjecture by Yau, which is one of the main open problems in this topic.

We start with a general result on k\"ahlerian manifolds \cite{AGF1}, that builds in the seminal work of Li and Yau in \cite{LiYau}.

\begin{theorem}[\cite{AGF1}]\label{thm:existence0}
Let $(X,\Omega)$ be a compact Calabi-Yau threefold endowed with a K\"ahler Ricci-flat metric $\omega_\infty$ with holonomy $SU(3)$. Let $\cE$ be a holomorphic vector bundle over $X$ satisfying \eqref{eq:c1} and \eqref{eq:c2BC}. If $\cE$ is stable with respect to $[\omega_\infty^2]$, then there exists a $1$-parameter family of solutions $(h_\delta,\omega_\delta,\dbar_\delta)$ of the Strominger system \eqref{eq:Stromingersystem} such that $\frac{\omega_\delta}{\delta}$ converges to $\omega_\infty$ as $\delta \to \infty$.
\end{theorem}

To explain the main idea, we note that the Strominger system is invariant under rescaling of the hermitian form $\omega$, except for the Bianchi identity. Given a positive real constant $\delta$, if we change $\omega\to\delta\omega$ and define $\epsilon:= \alpha/\delta$ we obtain a new system, with all the equations unchanged except for the Bianchi identity, which reads
$$
dd^c\omega - \epsilon(\tr(R_\nabla\wedge R_\nabla)-\tr(F_A\wedge F_A)) =0.
$$
In the \emph{large volume limit} $\delta\to\infty$, a solution of the system is given by prescribing degree zero stable holomorphic vector bundles $\cE$ and $\cT$ over a Calabi-Yau threefold with hermitian metric $\omega_\infty$ satisfying
$$
d(\|\Omega\|_{\omega_\infty}\omega_\infty^2) = 0, \qquad dd^c \omega_\infty =0.
$$
The combination of these two conditions implies that $\omega_\infty$ is actually K\"ahler Ricci-flat, and by the Hermite-Yang-Mills condition for $\nabla$ we also have that $\cT \cong TX$ (see \cite[Lem. 4.1]{AGF1}). The final step is to perturb a given solution with $\epsilon=0$ to a solution with small $\epsilon > 0$, that is,
with large $\delta$, provided that \eqref{eq:c2BC} is satisfied. This is done via the Implicit Function Theorem in Banach spaces. The perturbation leaves the holomorphic structure of $\cE$ unchanged while the one on $TX$ is shifted by a complex gauge transformation and so remains isomorphic to the initial one.

We give a concrete example where the hypothesis of Theorem \ref{thm:existence0} are satisfied. For further examples of stable bundles on algebraic Calabi-Yau threefolds satisfying \eqref{eq:c1} and \eqref{eq:c2BC} we refer to \cite{AGF1,AGF2,Jardim}.

\begin{example}
For $X$ a generic quintic in $\mathbb{CP}^4$ (see Example \ref{ex:Candelas}), any K\"ahler Ricci-flat metric has holonomy $SU(3)$. Then, by a result of Huybrechts \cite{Huyb}, the bundle $TX \oplus \mathcal{O}_X$ admits stable holomorphic deformations $\cE$, which therefore have the same Chern classes as $TX$. The application of Theorem \ref{thm:existence0} in this example recovers \cite[Th. 5.1]{LiYau}.
\end{example}

We recall next the reduction method of Fu and Yau \cite{FuYau}, based on the non k\"ahlerian fibred threefolds constructed by Goldstein and Prokushkin \cite{GoPro}. This result does not impose the Hermite-Yang-Mills condition on $\nabla$, that is taken to be the Chern connection of the hermitian metric on $X$. Let $(S,\Omega_S)$ be a $K3$ surface with a K\"ahler Ricci-flat metric $g_S$ and K\"ahler form $\omega_S$. Let $\omega_1$ and $\omega_2$ be anti-self-dual $(1,1)$-forms on $S$ such that 
$$
[\omega_i/2\pi] \in H^2(S,\ZZ).
$$
Let $X$ be the total space of the fibred product of the $U(1)$ line bundles determined by $[\omega_1/2\pi]$ and $[\omega_2/2\pi]$. Given a function $u$ on $S$, consider the hermitian form
\begin{equation}\label{eq:FuYau}
\omega_u = p^*(e^u \omega_S) + \frac{i}{2} \theta \wedge \overline{\theta},
\end{equation}
where $\theta$ is a connection on $X$ such that $i F_\theta = \omega_1 + \omega_2$, and the complex threeform
$$
\Omega = \Omega_S \wedge \theta.
$$
Then, using that $\omega_1$ and $\omega_2$ are anti-self-dual, it is easy to check that $\omega_u$ satisfies the dilatino equation \eqref{eq:dilatinosec2} and $d \Omega = 0$. Let $\cE_S$ be a degree zero $[\omega_S]$-stable holomorphic vector bundle over $S$. Define $\cE = p^*\cE_S$ and $h = p^*h_S$, where $h_S$ is the Hermite-Einstein metric on $\cE_S$. Then, $h$ is a Hermite-Einstein metric for $\omega_u$ and hence with this ansatz the Strominger system reduces to the Bianchi identity. This identity is actually equivalent to the following complex Monge-Amp\`ere equation on $S$
$$
dd^c (e^u \omega - \alpha e^{-u}\rho) + \frac{1}{2}dd^c u \wedge dd^c u = \mu \omega_S^2/2,
$$
where $\rho$ is a smooth real $(1,1)$-form on $S$ independent of $u$ and
$$
\mu \omega_S^2 = (|\omega_1|^2 + |\omega_2|^2)\omega_S^2 + \alpha (\tr F_h \wedge F_h - R_{\omega_S} \wedge R_{\omega_S}).
$$
Here, $R_{\omega_S}$ denotes the curvature of the Chern connection of $\omega_S$ on $S$.

\begin{theorem}[\cite{FuYau}]\label{thm:FuYau}
The equation \eqref{eq:FuYau} admits a solution for $\alpha > 0$, provided that
$$
0 = \int_S \mu \omega_S^2 = \int_S (|\omega_1|^2 + |\omega_2|^2)\omega_S^2 - 8\pi^2 \alpha (24 - c_2(\cE_S)).
$$ 
\end{theorem}

The case $\alpha < 0$ was proved in \cite{FuYau2}, and more recently in \cite{Phongslope} with different methods.

We consider now the method of invariant solutions in homogeneous spaces. Following \cite{FeiYau,OUV}, we describe an explicit solution of the form $X = \SL(2,\CC)/\Gamma$, for $\Gamma$ a cocompact lattice in $\SL(2,\CC)$. The group $\SL(2,\CC)$ is unimodular and therefore it admits a biinvariant holomorphic volume form. Furthermore, any right invariant metric on $\SL(2,\CC)$ is balanced, and solves the dilatino equation (cf. Example \ref{example:balanced2}). With the ansatz $F_A = 0$ for the connection $A$, the Strominger system \eqref{eq:Stromingersystem2} reduces to the conditions
\begin{equation}\label{eq:StromingerSL}
\begin{split}
\Lambda_\omega R_\nabla & = 0, \qquad R_\nabla^{0,2} = 0 \\
dd^c \omega - \alpha \(\tr R_\nabla \wedge R_\nabla\) & = 0. 
\end{split}
\end{equation} 
We want to check that the system is satisfied for $\nabla = \nabla^g - \frac{1}{2}g^{-1}d^c\omega$ the Bismut connection of $\omega$. To see this, consider a right invariant basis $\{\sigma^1,\sigma^2,\sigma^3\}$ of $(1,0)$-forms satisfying
$$
d \sigma^1 = \sigma^2 \wedge \sigma^3, \qquad d \sigma^2 = - \sigma^1 \wedge \sigma^3, \qquad d \sigma^3 = \sigma^1 \wedge \sigma^2.
$$
Consider the biinvariant holomorphic volume form
$$
\Omega = \sigma^1 \wedge \sigma^2 \wedge \sigma^3
$$
and the right invariant hermitian metric
$$
\omega_t = \frac{i}{2}t^2 (\sigma^1 \wedge \overline{\sigma}^1 + \sigma^2 \wedge \overline{\sigma}^2 + \sigma^3 \wedge \overline{\sigma}^3)
$$
for $t \in \RR \backslash \{0\}$. Define a real basis of right invariant $1$-forms by
$$
e^1 + i e^2 = t \sigma^1, \qquad e^3 + i e^4 = t \sigma^2, \qquad e^5 + i e^6 = t \sigma^3.
$$
Then, a direct calculation shows that (see \cite[Th. 4.3]{OUV} and \cite[eq. (8)]{FeiYau}
$$
dd^c \omega_t = -\frac{4}{t^2}(e^{1234} + e^{1256} + e^{3456}),
$$
and for $\nabla$ the Bismut connection
$$
\alpha \tr R_\nabla \wedge R_\nabla = - \alpha \frac{16}{t^4}(e^{1234} + e^{1256} + e^{3456}).
$$
Therefore, for $\alpha > 0$ taking $t$ such that $\alpha = t^2/4$ we obtain a solution of the Bianchi identity. Furthermore, by \cite[Prop. 4.1]{OUV} the Bismut connection is also a solution of the Hermite-Yang-Mills equations.

Although the three methods we have just explained provide a large class of solutions of the Strominger system in complex dimension $3$, the existence problem is widely open. The following conjecture by Yau is one of the main open problems in this topic.

\begin{conjecture}[Yau \cite{Yau2010}]\label{conj:Yau}
Let $(X,\Omega)$ be a compact Calabi-Yau threefold endowed with a balanced class $\tau$. Let $\cE$ be a holomorphic vector bundle over $X$ satisfying \eqref{eq:c1} and \eqref{eq:c2BC}. If $\cE$ is stable with respect to $\tau$, then $(X,\Omega,\cE)$ admits a solution of the Strominger system.
\end{conjecture}

Even for k\"ahlerian manifolds, Conjecture \ref{conj:Yau} is not completely understood. In this setup, Theorem \ref{thm:existence0} provides a solution of Conjecture \ref{conj:Yau} for balanced classes of the form $\tau = [\omega]^2$, where $[\omega]$ is a K\"ahler class on $X$. We note however that Fu and Xiao \cite{FuXiao} have proved that for projective Calabi-Yau $n$-folds the cohomology classes
$$
[\beta] \in H^{1,1}(X,\RR)
$$   
such that $[\beta]^n > 0$ -- known as \emph{big classes} -- satisfy that $[\beta]^{n-1}$ is a balanced class. An interesting example of a big class which is not K\"ahler is provided by Example \ref{ex:Candelas} on conifold transitions. With the notation stated there, if $L$ is the pull-back of any ample divisor on $X_0$, then $c_1(L)$ is a big (and nef) class on $X$. By a result of Tosatti \cite{Tosatti}, the smooth Ricci-flat metrics on $X$ with classes approaching $c_1(L)$ have a well-defined limit, given by the pull-back of the unique singular Ricci-flat metric on $X_0$. It is plausible that this result combined with Theorem \ref{thm:existence0} can be used to prove Conjecture \ref{conj:Yau} for algebraic Calabi-Yau threefolds. We should also note that the method of Theorem \ref{thm:existence0} does not have any control on the balanced class of the final solution. On general grounds, it is expected that the solution predicted by Conjecture \ref{conj:Yau} has balanced class $\tau$.

For non-k\"ahlerian manifolds, Yau's Conjecture is widely open. To illustrate this, we state a basic question that should be addressed before dealing with the more general Conjecture \ref{conj:Yau}.

\begin{question}\label{question}
Let $X$ be a compact complex manifold with balanced class $\tau \in H^{n-1,n-1}_{BC}(X,\RR)$. Let $\rho \in \Omega^{n-1,n-1}$ be a real $dd^c$-exact form on $X$. Is there a balanced metric $g$ on $X$ with balanced class $\tau$ solving the following equation?
\begin{equation}\label{eq:Bianchiputative}
dd^c \omega = \rho.
\end{equation}
\end{question}

A promising approach to Conjecture \ref{conj:Yau} using geometric flows -- which in particular treats a question closely related to Question \ref{question} -- has been recently proposed in \cite{Phong}.

Setting Question \ref{question} in the affirmative in the case of Clemens-Friedman non-k\"ahlerian complex manifolds would provide important support for Yau's proposal of a metric approach to Reid's Fantasy (see Section \ref{sec:conifoldtr}). Equation \eqref{eq:Bianchiputative} (as well as the Strominger system) pins down a particular solution of the dilatino equation \eqref{eq:dilatinosec2} in a given balanced class, via a condition on the torsion $-d^c\omega$ of the Bismut connection. As we have pointed out earlier in this section, when $\rho = 0$ the combination of \eqref{eq:Bianchiputative} with the dilatino equation \eqref{eq:dilatinosec2} is equivalent to the metric being Calabi-Yau. The mechanism whereby a conifold transition creates a $4$-form $\rho$ which couples to the metric is still unknown (in physics, $\rho$ can be interpreted as the Poincar\'e dual four-current of a holomorphic submanifold wrapped by a NS$5$-brane \cite{TsengYau}).

\section{Physical origins and string classes}\label{sec:physics}

\subsection{The Strominger system in heterotic supergravity}\label{sec:sugra}

The Strominger system arises in the low-energy limit of the heterotic string theory. 
This theory is described by a $\sigma$-model, a quantum field theory with fields given by smooth maps $C^\infty(S,N)$, from a smooth surface $S$ -- the worldsheet of the string -- into a target manifold $N$. From the point of view of the worldsheet, the theory leads to a superconformal field theory. In the low-energy limit, the heterotic string can be described from the point of view of $N$, yielding a supergravity theory. We start with a discussion of classical heterotic supergravity, which allows a rigorous derivation of the Strominger system (and completely omits perturbation theory). We postpone a conceptual explanation of the Bianchi identity 
to the next section, where we adopt the worldsheet approach.

Heterotic supergravity is a ten-dimensional supergravity theory coupled with super Yang-Mills theory (see e.g. \cite[p. 1101]{QFS}). It is formulated on a $10$-dimensional spin manifold $N$, i.e. oriented, with vanishing
second Stiefel-Whitney class
$$
w_2(N) = 0,
$$
and with a choice of element in $H^1(N,\mathbb Z_2)$. The manifold is endowed with a principal bundle $P_K$, with compact structure group $K$, contained in $SO(32)$ or $E_8 \times E_8$. We will assume $K = SU(r)$.

The (bosonic) field content of the theory is given by a metric $g_0$ of signature $(1,9)$ (in the string frame), a (dilaton) function $\phi \in C^\infty(N)$, a $3$-form $H \in \Omega^3$ and a (gauge) connection $A$ on $P_K$. We ignore the fermionic fields in our discussion. The equations of motion can be written as
\begin{equation}\label{eq:motionHetsugra}
  \begin{split}
    \operatorname{Ric}^{g_0} - 2 \nabla^{g_0}(d\phi) - \frac{1}{4} H \circ H + \alpha \tr F \circ F - \alpha \tr R \circ R & = 0,\\
    d^*(e^{2\phi}H) & = 0,\\
    d^*_A(e^{2\phi}F) + \frac{e^{2\phi}}{2} *(F \wedge *H) & = 0,\\
    S^{g_0} - 4 \Delta \phi - 4 |d\phi|^2 - \frac{1}{2}|H|^2 + \alpha(|R|^2
- |F|^2) & = 0,
  \end{split}
\end{equation}
where $\alpha$ as positive real constant -- the \emph{slope string parameter} --, $F$ is the curvature of $A$,  and $R$ is the curvature of an
auxiliary connection $\nabla_0$ on $TN$. Here, $H \circ H$ is a symmetric $2$-tensor constructed by contraction with the metric (and similarly for $\tr F \circ F - \tr R \circ R$)
$$
(H \circ H)_{mn} = g_0^{ij}g_0^{kl}H_{ikm} H_{jln}.
$$

The introduction of the connection $\nabla_0$ -- which is not considered as a physical field -- is due to the cancellation of anomalies (this is the failure of a classical symmetry to be a symmetry of the quantum theory). The Green-Schwarz mechanism of anomaly cancellation \cite{GreenSchwarz} sets a particular local ansatz for the three-form
\begin{equation}\label{eq:Hsugra}
  H = db - \alpha(CS(\nabla_0) - CS(A)),
\end{equation}
in terms of a $2$-form $b \in \Omega^2$ and the Chern-Simons $3$-forms of $A$ and $\nabla_0$. We use the convention $d CS(A) = -\tr F_A \wedge F_A$. Up to the Chern-Simons term, $H$ can be regarded as the \emph{field strength} of the locally-defined \emph{$B$-field} $b$. Although \eqref{eq:Hsugra} fails to be globally well-defined on $N$, it imposes the global \emph{Bianchi identity} constraint
\begin{equation}\label{eq:dHsugra}
  dH = \alpha(\tr R \wedge R - \tr F \wedge F).
\end{equation}
We postpone the conceptual explanation of this equation to the worldsheet approach.

We note that the equations \eqref{eq:motionHetsugra} do not arise as critical points of any functional, e.g. due to the term $*(F \wedge *H)$ in the third equation. Rather, physicists consider the \emph{pseudo-action} 
$$
\int_N e^{-2\phi}(S^{g_0} + 4|d\phi|^2 - \frac{1}{12}|H|^2 +
\frac{\alpha'}{2}(|R|^2- |F|^2))\operatorname{Vol}_{g_0},
$$
where the norm squared of $F$ and $R$ is taken with respect to the
Killing form $-\tr$. Remarkably, calculating the critical points of this functional and taking the local form \eqref{eq:Hsugra} of $H$ into account, yields the equations of motion \eqref{eq:motionHetsugra}.

In supergravity theories, supersymmetry distinguishes special solutions of the equations of motion which are fixed by the action of a (super) Lie algebra on the space of fields. Generators for this action are typically given in terms of spinors. In the case of heterotic supergravity, ($N=1$) supersymmetry requires the existence of a non-vanishing Majorana-Weyl spinor $\epsilon$ with positive chirality (see \cite[p. 9]{Figueroa}), satisfying
the \emph{Killing spinor equations}
\begin{equation}\label{eq:susy10d}
  \begin{split}
    F \cdot \epsilon & = 0\\
    \nabla^- \epsilon & = 0,\\
    (H + 2d\phi) \cdot \epsilon & = 0,
  \end{split}
\end{equation}
Here $\nabla^-$ is the metric connection with skew torsion $-H$
obtained from the Levi-Civita connection $\nabla^{g_0}$
\begin{equation}\label{eq:nablaminus}
\nabla^- = \nabla^{g_0} - \frac{1}{2}g_0^{-1}H.
\end{equation}

The relation between the heterotic supergravity equations and the Strominger system arises via a mechanism called \emph{compactification}, whereby the $10$-dimensional theory is related to a theory in $4$-dimensions. Strominger \cite{Strom} and Hull \cite{HullTurin} characterized the geometry of a very general class of compactifications of heterotic supergravity, inducing a $4$-dimensional supergravity theory with $N=1$ supersymmetry. The geometric conditions that they found in the so called \emph{internal space} is what is known today as the Strominger system. Mathematically, Strominger-Hull compactifications amount to the following ansatz: the space-time manifold is a product
$$
N = \RR^4 \times M,
$$
where $M$ is a compact smooth oriented spin $6$-dimensional manifold -- the internal space --, with metric given by
$$
g_0 = e^{2(f - \phi)}(g_{1,3} \times g)
$$ 
for $g_{1,3}$ a flat Lorentz metric, $g$ a riemannian metric on $M$ and $f \in C^\infty(M)$ a smooth function on $M$. The fields $H$ and $A$ are pull-back from $M$, and $\nabla_0$ is the product of the Levi-Civita connection of $g^{1,3}$ with a connection $\nabla$ on $TM$ compatible with $g$. The condition of $N =1$ supersymmetry in $4$-dimensions imposes that $f = \phi$ and also
$$
\epsilon = \zeta \otimes \eta + \zeta^* \otimes \eta^*,
$$ 
where $\zeta$ is a positive chirality spinor for $g_{1,3}$ and $\eta$ is a positive chirality spinor for $g$ (living in complex representations of the corresponding real $\operatorname{Spin}$ group), while $\zeta^*$ and $\eta^*$ denote their respective conjugates. With this ansatz, the equations \eqref{eq:motionHetsugra}, \eqref{eq:susy10d} and \eqref{eq:dHsugra} are equivalent, respectively, to equations for $(g,\phi,H,\nabla,A,\eta)$ on $M$:
\begin{equation}\label{eq:motion6d}
  \begin{split}
    \operatorname{Ric}^g - 2 \nabla^g(d\phi) - \frac{1}{4} H \circ H - \alpha F_A \circ F_A + \alpha R_\nabla \circ R_\nabla & = 0,\\
    d^*(e^{2\phi}H) & = 0,\\
    d^*_A(e^{2\phi}F_A) + \frac{e^{2\phi}}{2} *(F_A \wedge *H)& = 0,\\
S^{g} - 4 \Delta \phi - 4 |d\phi|^2 - \frac{1}{2}|H|^2 + \alpha(|R_\nabla|^2
- |F_A|^2) & = 0
  \end{split}
\end{equation}

\begin{equation}\label{eq:susy6d}
  \begin{split}
    \nabla^- \eta & = 0,\\
    (d\phi + \frac{1}{2}H) \cdot \eta & = 0,\\
    F_A \cdot \eta & = 0,\\
  \end{split}
\end{equation}

\begin{equation}\label{eq:dH6d}
    dH - \alpha(\tr R_\nabla \wedge R_\nabla - \tr F_A \wedge F_A) = 0.
\end{equation}

The following characterization of \eqref{eq:susy6d} and \eqref{eq:dH6d} in terms of $SU(3)$-structures is due to Strominger and Hull.

\begin{theorem}[\cite{HullTurin,Strom}]\label{th:strom}
A solution $(g,\phi,H,\nabla,A,\eta)$ of \eqref{eq:susy6d} and \eqref{eq:dH6d} is equivalent to a Calabi-Yau structure $(X,\Omega)$ on $M$, with hermitian metric $g$ and connection $A$ on $P_K$ solving
\begin{equation}\label{eq:Stromingersystemoriginal}
\begin{split}
\Lambda_\omega F_A & = 0, \qquad F_A^{0,2} = 0, \\
d^*\omega - d^c \log \|\Omega\|_\omega & = 0,\\
dd^c \omega - \alpha \(\tr R_\nabla \wedge R_\nabla - \tr F_A \wedge F_A\) & = 0. 
\end{split}
\end{equation}
where
$$
H = d^c \omega, \qquad d \phi = - \frac{1}{2} d \log \|\Omega\|_\omega.
$$
\end{theorem}

For the proof, one can use that the stabilizer of $\eta$ in $\operatorname{Spin}(6)$ is $\SU(3)$, and write the equations in terms of the corresponding $\SU(3)$-structure. The key point is that, since $\nabla^-$ is unitary and has totally skew-torsion $-H$, by \cite[Eq. (2.5.2)]{Gau1},
$$
H = -N + (d^c\om)^{2,1+1,2},
$$
where $N$ denotes the Nijenhuis tensor of the almost complex structure determined by $\eta$. We refer to \cite[Th. 6.10]{GRT} for a detailed proof of this result. We note that the same result holds on an arbitrary even-dimensional manifold $M$, provided that the spinor $\eta$ is pure (see \cite[Lem. 9.15]{LaMi} and \cite[Rem. 9.12]{LaMi}).

Supersymmetric vacuum of heterotic supergravity compactified on $M$ -- with the Strominger-Hull ansatz -- correspond to solutions of the system of equations formed by \eqref{eq:motion6d}, \eqref{eq:susy6d} and \eqref{eq:dH6d}. Therefore, finding a solution of \eqref{eq:Stromingersystemoriginal} is a priori not enough to find a supersymmetric classical solution of the theory. 
This problem was understood by Fernandez, Ivanov, Ugarte and Villacampa \cite{FIVU,Ivan09}, who provided a characterization of the solutions of the Killing spinor equations \eqref{eq:susy6d} and the Bianchi identity \eqref{eq:dH6d} which also solve the equations of motion \eqref{eq:motion6d}.

\begin{theorem}[\cite{FIVU,Ivan09}]\label{th:FIVU}
A solution of the Killing spinor equations \eqref{eq:susy6d} and the Bianchi identity \eqref{eq:dH6d} is a solution of the equations of motion \eqref{eq:motion6d} if and only if
  \begin{equation}\label{eq:Rinst1}
    R_\nabla \cdot \epsilon = 0.
  \end{equation}
\end{theorem}

The \emph{instanton condition} \eqref{eq:Rinst1} is equivalent to $\nabla$ being a Hermite-Yang-Mills connection, that is,
\begin{equation}\label{eq:Rinst2}
  R_\nabla^{0,2}=0, \ \ \ R_\nabla \wedge \omega^2=0.
\end{equation}
Combined with Theorem \ref{th:strom}, Theorem \ref{th:FIVU} establishes the link between the equations in heterotic supergravity and the Strominger system \eqref{eq:Stromingersystem2}. To state a precise result, which summarizes the previous discussion, we observe that it follows from the proof of Theorem \ref{th:FIVU} that the instanton condition $R_\nabla \cdot \epsilon = 0$ jointly with \eqref{eq:susy6d} and \eqref{eq:dH6d} implies the following \emph{equation of motion for $\nabla$}
$$
d^*_\nabla(e^{2\phi}R_\nabla) + \frac{e^{2\phi}}{2} *(R_\nabla \wedge *H) = 0.
$$

\begin{theorem}\label{th:strom2}
A solution $(g,\phi,H,\nabla,A,\eta)$ of the system
\begin{equation}\label{eq:heterotic6d}
  \begin{split}
    \nabla^- \eta & = 0,\\
    (d\phi + \frac{1}{2}H) \cdot \eta & = 0,\\
    F_A \cdot \eta & = 0,\\
    R_\nabla \cdot \eta & = 0,\\
    dH - \alpha(\tr R_\nabla \wedge R_\nabla - \tr F_A \wedge F_A) & = 0.
\end{split}
\end{equation}
is equivalent to a Calabi-Yau structure $(X,\Omega)$ on $M$, with hermitian metric $g$ and connection $A$ on $P_K$ solving the Strominger system \eqref{eq:Stromingersystem2}, where
$$
H = d^c \omega, \qquad d \phi = - \frac{1}{2} d \log \|\Omega\|_\omega.
$$
Furthermore, any solution of \eqref{eq:heterotic6d} solves the equations of motion
\begin{equation}\label{eq:motion6d-2}
  \begin{split}
    \operatorname{Ric}^g - 2 \nabla^g(d\phi) - \frac{1}{4} H \circ H - \alpha F_A \circ F_A + \alpha R_\nabla \circ R_\nabla & = 0,\\
    d^*(e^{2\phi}H) & = 0,\\
    d^*_A(e^{2\phi}F_A) + \frac{e^{2\phi}}{2} *(F_A \wedge *H)& = 0,\\
    d^*_\nabla(e^{2\phi}R_\nabla) + \frac{e^{2\phi}}{2} *(R_\nabla \wedge *H)& = 0,\\
    S^{g} - 4 \Delta \phi - 4 |d\phi|^2 - \frac{1}{2}|H|^2 + \alpha(|R_\nabla|^2
- |F_A|^2) & = 0.
\end{split}
\end{equation}

\end{theorem}

The formal symmetry of the equations \eqref{eq:heterotic6d} in the connections $\nabla$ and $A$ (which flips a sign in the Bianchi identity), seems to be crucial for the understanding of the geometry and the moduli problem for the Strominger system, that we review in Section \ref{sec:moduli}.

\begin{remark}\label{rem:Stromoriginal}
We note that in the formulation of the Strominger system in \cite{Strom} the condition \eqref{eq:Rinst2} was not included, probably relying on the general principle that supersymmetry implies the equations of motion of the theory (which is not valid for the heterotic string even in perturbation theory \cite{OssaSvanes2}). The analysis in \cite{HullTurin,Hull2} takes this fact into account and proposes $\nabla = \nabla^+$ (given by changing $H \to -H$ in \eqref{eq:nablaminus}) as the preferred connection to solve \eqref{eq:dH6d}. From a mathematical perspective, this last statement has to be taken rather formally. It can be regarded as a perturbative version of Theorem \ref{th:FIVU}, in the following sense: expanding $g$ and $H$ in a formal power series in the parameter $\alpha$ it follows that
$$
R_{\nabla^+} \cdot \eta = O(\alpha),
$$
provided that $\nabla^- \eta = 0$ is satisfied (see e.g. \cite[App. C]{OssaSvanes2}).
\end{remark}


\subsection{The worldsheet approach and string classes}\label{sec:worldsheet}

A more fundamental approach to heterotic string theory is provided by the (non-linear) $\sigma$-model. We start with a brief (and rather naive) description of this theory, in order to explain the Bianchi identity in heterotic supergravity and its relation with the notion of \emph{string class}. 

The heterotic non-linear $\sigma$-model is a two-dimensional quantum field theory with (bosonic) fields given by smooth maps $f \in C^\infty(S,N)$ from an oriented surface $S$ into a target manifold $N$. To describe the classical action, we fix a metric $\gamma$ on $S$ -- with volume $\Vol_{\gamma}$ and scalar curvature $R_{\gamma}$ -- and \emph{background fields} $(g^0,\phi,b,A)$ on $N$. Here, $(g^0,\phi,A)$ are as in the previous section, and $b$ is a \emph{$B$-field}, that by now we treat as a (local) two-form on $N$ (more invariantly, it will correspond to a trivialization of a bundle $2$-gerbe with connection \cite{Waldorf}). The action is
\begin{equation}\label{eq:actionsigma}
\frac{1}{4\pi \alpha}\int_S |df|^2 \Vol_{\gamma} + f^*b - \frac{\alpha}{2}\phi R_{\gamma} \Vol_{\gamma} + \ldots
\end{equation}
where $|df|^2$ denotes the norm square of $df$ with respect to $\gamma$ and $g$. The terms denoted $\ldots$ correspond to the fermionic part of the action -- depending on the connections $A$ and $\nabla^- = \nabla^g - \frac{1}{2}g^{-1}db$, and a choice of spin structure on $S$ -- that we omit for simplicity (see e.g. \cite{Strom}).

The constant $\alpha$ in \eqref{eq:actionsigma} is $2 \ell^2$, where $\ell$ is the Planck length scale 
(and hence positive). The background fields in the $\sigma$-model appear as coupling functions (generalizing the notion of coupling constant). The value of the dilaton $\phi$ at a point determines the string coupling constant, i.e. the strength with which strings interact with each other. The dilaton is a special field, as the term
$$
\int_S \phi R_{\gamma} \Vol_{\gamma}
$$
destroys the conformal invariance of the action \eqref{eq:actionsigma} (classical Weyl invariance). Nonetheless, the inclusion of this term is crucial for the conformal invariance of the theory at the quantum level \cite{CFMP}.

The quantum theory constructed from the action \eqref{eq:actionsigma} is defined perturbatively, in an expansion in powers of $\alpha$. Conformal invariance of the \emph{effective action} corresponds to the vanishing of the \emph{$\beta$-functions}, which in the critical dimension $\dim N = 10$ are given by
\begin{equation}\label{eq:betafunctions}
\begin{split}
\beta^G & = \operatorname{Ric}^{g_0} - 2 \nabla^{g_0}(d\phi) - \frac{1}{4} H \circ H + \alpha \tr F \circ F - \alpha \tr R \circ R + O(\alpha^2),\\
\beta^B & = d^*(e^{2\phi}H)  + O(\alpha^2),\\
\beta^A & = d^*_A(e^{2\phi}F) + \frac{e^{2\phi}}{2} *(F \wedge *H)  + O(\alpha^2),\\
\beta^\phi & = S^{g_0} - 4 \Delta \phi - 4 |d\phi|^2 - \frac{1}{2}|H|^2 + \alpha(|R|^2 - |F|^2)  + O(\alpha^2),
\end{split}
\end{equation}
where $H$ is a three-form on $N$ and $R$ is the curvature of an auxiliary connection $\nabla_0$ on $TN$, locally related with $b$ and $A$ by the Green-Schwarz ansatz \eqref{eq:Hsugra}. The sudden appearance of the extra connection $\nabla_0$ is explained by the way the fields are treated in perturbation theory, as formal expansions in the parameter $\alpha$: even though the connection in \eqref{eq:Hsugra} in the perturbation theory analysis is $\nabla^+$ \cite{HullTownsend}, the truncation to second order in $\alpha$-expansion enables to remove its dependence from the rest of fields (cf. Remark \ref{rem:Stromoriginal}). 

We observe that the classical equations of motion of the heterotic supergravity in the target \eqref{eq:motionHetsugra} are given by first-order conditions for conformal invariance (in $\alpha$-expansion) of the quantum theory in the worldsheet of the string. Similarly, the killing spinor equations \eqref{eq:susy10d} are obtained as first order conditions in $\alpha$-expansion in order to define a supersymmetric theory \cite{Hull2,Strom}. When the $\beta$-functions vanish for a choice of background fields, the heterotic $\sigma$-model is expected to yield a two-dimensional superconformal field theory (with $(0,2)$-supersymmetry, when the killing spinor equations are satisfied). Unfortunately, to the present day a closed form of \eqref{eq:betafunctions} to all orders in $\alpha$-expansion is unknown. Despite of this fact, several rigorous attempts to construct this theory by indirect methods can be found in the literature (see e.g. \cite{Donagi,GSVI,MSV}).

Aside from the disturbing perturbation theory, the (fermionic) terms omitted in the classical action \eqref{eq:actionsigma} provide a conceptual explanation of the Bianchi identity \eqref{eq:dHsugra}, that we have ignored so far. The path integral quantization of the $\sigma$-model yields a (determinant) line bundle $\cL$ over a space of Bose fields 
$$
B = \textrm{Conf}(S) \times C^\infty(S,N)/\Diff_0(S),
$$
where $\textrm{Conf}(S)$ is the space of conformal structures on $S$ and $\Diff_0(S)$ is the identity component of the diffeomorphism group of $S$. There is a canonical section $s_0$ of $\cL$ -- determined by the fermionic terms in the action -- that should be integrated over $B$, and hence one tries to find a trivialization of $\cL$, so as to express $s_0$ as a function on $B$. The obstruction to finding a trivialization is called the \emph{anomaly}. This connection between anomalies and determinant line bundles was pioneered by Atiyah and Singer, in close relation to the Index Theorem. Here we follow closely a refined geometric version by Witten \cite{WittenCMP} and Freed \cite{Freed}. 

The construction of the line bundle $\cL$ in \cite{Freed} assumes the compactification ansatz $N = \RR^4 \times M^6$, discussed in the previous section. Further, $M$ is endowed with an integrable almost complex structure $J$ compatible with $g$, such that $c_1(TM,J) = 0$. Let $(E,h)$ be a smooth hermitian vector bundle over $M$. We assume in this section that $c_1(E) = 0$ and that $E$ has rank $16$. The aim is to give an explanation of the \emph{topological Bianchi identity}
\begin{equation}\label{eq:bianchitop}
dH = \alpha \(\tr R_\nabla \wedge R_\nabla - \tr F_A \wedge F_A\),
\end{equation} 
as an equation for an arbitrary three-form $H \in \Omega^3$ on $M$ and a pair of special unitary connections $\nabla$ on $(TM,J)$ and $A$ on $(E,h)$.  Observe that \eqref{eq:bianchitop} implies a condition in the real first Pontryagin class $p_1(M)$ of $M$, namely
\begin{equation}\label{eq:c2dR}
p_1(E) = p_1(M) \in H^{4}(M,\RR).
\end{equation}

Given the data $(g,\nabla,A)$, Freed constructs in \cite{Freed} a (Pfaffian) complex line bundle
$$
\cL \to B,
$$
which is trivializable provided that
\begin{equation}\label{eq:p1dRZ}
\frac{1}{2}p_1(E) = \frac{1}{2}p_1(M) \in H^{4}(M,\ZZ).
\end{equation}
Furthermore, this line bundle is endowed with a natural connection $\mathbb{A}^0$ on $\cL$ whose curvature $F_{\mathbb{A}^0}$ can be identified with
$$
4\pi i F_{\mathbb{A}^0} \equiv \tr R_\nabla \wedge R_\nabla - \tr F_A \wedge F_A.
$$
Note here that a $4$-form $\rho$ on $Y$ can be regarded as a $2$-form $\psi$ on $C^\infty(S,N)$, by
$$
\psi(V_1,V_2) = \int_S f^* \iota_{V_1}\iota_{V2} \rho,
$$
where $f \in C^\infty(S,N)$ and $V_1,V_2 \in T_f C^\infty(S,N)$, where the tangent space at $f$ is identified with $C^\infty(S,f^*TN)$.

Assuming that \eqref{eq:p1dRZ} is satisfied, we try to parametrize flat connections with trivial holonomy on $\cL$.
This is an important question from a physical perspective, as different trivializations of $\cL$ correspond to different partition functions of the heterotic string theory. 
The answer is closely related to the notion of \emph{string class} \cite{Redden}, that we introduce next. Let $P_g$ be the bundle of special unitary frames of the hermitian metric $g$ on $(TM,J)$, with structure group $\SU(3)$, and let $P_h$ be the bundle of special unitary frames of the hermitian metric $g$, with structure group $\SU(r)$. Consider the principal bundle $p \colon P \to M$ given by 
$$
P = P_g \times_M P_h,
$$
with structure group $G = \SU(3) \times \SU(r)$. Let $\sigma$ denote the (left) Maurer-Cartan $1$-form on $G$. We fix a biinvariant pairing on $\mathfrak{g}$
$$
c = \alpha (\tr_{\mathfrak{su}(3)} - \tr_{\mathfrak{u}(r)}).
$$
We assume that $c$ is suitably normalized so that the $[\sigma_3] \in H^3(G,\ZZ)$, where
$$
\sigma_3 = \frac{1}{6}c(\sigma,[\sigma,\sigma]).
$$ 

\begin{definition}[\cite{Redden}]\label{def:stringclass}
A \emph{string class} on $P$ is a class $[\hat H] \in H^3(P,\ZZ)$ such that the restriction of $[\hat H]$ to any fibre of $P$ yields the class $[\sigma_3] \in H^3(G,\ZZ)$.
\end{definition}

String classes form a torsor over $H^3(M,\ZZ)$, where the action is defined by pull-back and addition \cite[Prop. 2.16]{Redden}: 
$$
[\hat H] \to [\hat H] + p^*[H],
$$
where $p \colon P \to M$ is the canonical projection on the principal bundle $P$ and $[H] \in H^3(M,\ZZ)$. 

Flat connections with trivial holonomy on the line bundle $\cL$ were interpreted by Bunke in \cite{Bunke} as (very roughly) enriched representatives of a string class in $P$. Here we give a simple-minded version of his construction, by choosing special $3$-form representatives. Let $[\hat H] \in H^3(P,\ZZ)$ be a string class. For our choice of connection $\theta = \nabla \times A$ on $P$ we can take a $G$-invariant representative $\hat H \in \Omega^3(P)$ 
of the form (see Section \ref{sec:gg})
\begin{equation}\label{eq:hatH}
\hat H = p^* H - CS(\theta)
\end{equation}
for a choice of $3$-form $H \in \Omega^3(M)$, determined up to addition of an exact three-form, where
$$
CS(\theta) = - \frac{1}{6}c(\theta \wedge [\theta,\theta]) + c(F_\theta \wedge \theta).
$$
Using that $d \hat H = 0$, we obtain that $(H,\nabla,A)$ solves the topological Bianchi identity \eqref{eq:bianchitop}, since $d CS(\theta) = c(F_\theta \wedge F_\theta)$. In conclusion, up to addition of an exact three-form, a string class determines a preferred solution of the topological Bianchi identity for a fixed connection $\theta$.

We construct now a flat connection $\mathbb{A}^H$ on $\cL$ using this fact. The connection $\mathbb{A}^H$ is defined by modification of $\mathbb{A}^0$ as follows
$$
d_{\mathbb{A}^H} \log s = d_{\mathbb{A}^0} \log s - \frac{\alpha^{-1}}{4\pi i} \int_S f^* H,
$$
where the left hand side is evaluated at the point $[(x,f)] \in B$ for $f \in C^\infty(S,N)$. The curvature of $\mathbb{A}^H$ can be identified with
$$
4\pi i F_{\mathbb{A}^H} \equiv \tr R_\nabla \wedge R_\nabla - \tr F_A \wedge F_A - \alpha^{-1} dH,
$$
and therefore $\mathbb{A}^H$ is flat. In \cite[Th. 4.14]{Bunke}, it is proved that $\mathbb{A}^H$ admits a parallel unit norm section $s$, therefore providing a trivialization of $\cL$. Note here that if we chose a different three-form $H + db$ on $M$ to represent $[\hat H]$, then this corresponds to a gauge transformation
$$
s \to e^{\frac{\alpha^{-1}}{4\pi i} \int_S f^* b} s
$$
of the section $s$.

String classes were introduced in \cite{Redden} to parametrize \emph{string structures} up to homotopy. String structures emerged from two-dimensional supersymmetric field theories in work by Killingback and Witten, and several definitions have been proposed so far. Given a spin bundle $P$ over $M$, McLaughlin defines a string structure on $P$ as a lift of the structure group of the looped bundle $LP = C^\infty(S^1,P)$ over the loop space
$$
LM = C^\infty(S^1,M)
$$ 
from $\operatorname{LSpin}(k) = C^\infty(S^1,\operatorname{Spin}(k))$ to its universal Kac-Moody central extension. Stolz and Teichner interpreted a string structure as a lift of the structure group of $P$ from $\operatorname{Spin}(k)$ to a certain three-connected extension, the topological group $\operatorname{String}(k)$. For recent developments on this topic, in relation to the topological Bianchi identity \eqref{eq:bianchitop}, we refer the reader to \cite{Sati,SSS}.


\section{Generalized geometry and the moduli problem}\label{sec:moduli}

\subsection{The Strominger system and generalized geometry}\label{sec:gg}

In this section we review on recent developments on the geometry of the Strominger system, based on joint work of the author with Rubio and Tipler \cite{GRT}. As we will see, the interplay of the Strominger system with the notion of string class (see Definition \ref{def:stringclass}) leads naturally to an interesting relation with Hitchin's theory of generalized geometry \cite{Hit1}, proposed in \cite{GF}.

To start the discussion, we draw a parallel between the Strominger system and Maxwell equations in electromagnetism (cf. \cite{TsengYau}). The Maxwell equations in a $4$-manifold $Y$ take the form
\begin{align*}
dF & = 0\\
d*F & = j_e
\end{align*}
where $F$ is a two-form -- the electromagnetic field strength -- and $j_e$ is the three-form \emph{electric current}. The cohomology class $[F] \in H^2(Y,\RR)$ is known as the \emph{magnetic flux} of the solution. Whereas in classical electromagnetism the magnetic flux  is allowed to take an arbitrary value in $H^2(Y,\RR)$, in the quantum theory Dirac's law of charge/flux quantization implies that magnetic fluxes are constrained to live in a full lattice inside $H^2(Y,\RR)$, namely
$$
[F/2\pi] \in H^2(Y,\ZZ).
$$
This changes the geometric nature of the problem. A geometric model which implements flux quantization takes the electromagnetic field to be the $i/2\pi$ times the curvature of a connection $A$ on a $U(1)$ line bundle over $Y$.
More generally, fixing the class $[F/2\pi] \in H^2(Y,\RR)$ amounts to fix the isomorphism class of a Lie algebroid over $Y$, and we can regard $A$ as a `global splitting' of the sequence defining the Lie algebroid. 

Given a solution of the Strominger system \eqref{eq:Stromingersystem2}, the three-form $d^c \omega$ and the connections $\nabla$ and $A$ determine a solution of the Bianchi identity \eqref{eq:bianchitop}. Relying on the discussion in Section \eqref{sec:worldsheet}, this data determines a \emph{real string class} (the analogue of Definition \ref{def:stringclass} in real cohomology). The aim here is to understand how the global nature of the geometric objects involved in the Strominger system changes, upon fixing the real string class of the solutions.

We will consider a simplified setup, where the principal bundle considered in Definition \ref{def:stringclass} is not necessarily related to the tangent bundle of the manifold. Let $M$ be a compact spin manifold $M$ of dimension $2n$. Let $P$ be a principal bundle, with  structure group $G$. We assume that there exists a non-degenerate pairing $c$ on the Lie algebra $\mathfrak{g}$ of $G$ such that the corresponding first Pontryagin class of $P$ vanishes
\begin{equation}\label{eq:p1moduli}
p_1(P) = 0 \in H^4(M,\RR).
\end{equation}
Let $\cA$ denote the space of connections $\theta$ on $P$. We denote by $\Omega^3_0 \subset \Omega^3_\CC$ the space of complex $3$-forms $\Omega$ such that
$$
T^{0,1}:=\lbrace V\in TM \otimes \CC \;\vert\; \iota_V \Om =0 \rbrace
$$
determines an almost complex structure $J_\Omega$ on $M$, that we assume to be non-empty. Consider the parameter space
$$
\cP \subset \Omega_0^3 \times \cA \times \Omega^2,
$$
defined by
$$
\cP=\lbrace (\Om,\theta,\omega)\;\vert\; \om \text{ is }
J_\Om-\text{compatible} \rbrace.
$$
The points in $\cP$ are regarded as unknowns for the system of
equations
\begin{equation}\label{eq:stromnonabelian}
  \begin{split}
    d\Om & =0, \ \ \ \ \ \ \ \ \ \ \ \ \ \ d(\vert\vert \Om \vert \vert_\om \om^2)  =  0,\\
    F_\theta^{0,2} & = 0, \ \ \ \ \ \ \ \ \ \ \ \ \ \ \ \ \ \ \ F_\theta\wedge \omega^2 = 0  ,\\
    dd^c \om- c(F_\theta\wedge F_\theta) & = 0,
  \end{split}
\end{equation}
where $F_\theta$ denotes the curvature of $\theta$, given explicitly by
$$
F_\theta = d \theta + \frac{1}{2}[\theta,\theta] \in \Omega^2(\ad P),
$$
where $\theta$ is regarded as a $G$-invariant $1$-form in $P$ with
values in $\mathfrak{g}$ and the bracket is the one on the Lie
algebra. The induced covariant derivative on the bundle of Lie
algebras $\ad P = P \times_G \mathfrak{g}$ is
$$
\iota_V d_\theta r = [\theta^\perp V,r],
$$
which satisfies $d_\theta \circ d_\theta = [F_\theta,\cdot]$.
 
To see the relation with the Strominger system, consider the particular case that $P$ is the fibred product of the bundle of oriented frames of $M$ and an $SU(r)$-bundle, with 
\begin{equation}\label{eq:pairingc}
  c = \alpha(- \tr -  c_{\mathfrak{gl}}).
\end{equation}
Here, $c_{\mathfrak{gl}}$ is a non-degenerate invariant metric on $\mathfrak{gl}(2n,\RR)$, which extends the non-degenerate Killing form $-\tr$ on $\mathfrak{sl}(2n,\RR) \subset \mathfrak{gl}(2n,\RR)$. Then, solutions $(\Omega,\omega,\theta)$ of the system \eqref{eq:stromnonabelian} correspond to solutions of \eqref{eq:Stromingersystem2}, provided that $\theta$ is a product connection $\nabla \times A$ and $\nabla$ is compatible with the hermitian structure $(\Omega,\omega)$. The compatibility between $\nabla$ and $(\Omega,\omega)$ leads to some difficulties in the construction, that we shall ignore here.

Going back to the general case, following Definition \ref{def:stringclass} we denote
$$
H^3_{str}(P,\RR) \subset H^3(P,\RR)
$$
the set of \emph{real string classes} in $P$. By condition \eqref{eq:p1moduli} this set is non-empty, and it is actually a torsor over $H^3(M,\RR)$. We note that any solution $x = (\Omega,\omega,\theta) \in \cP$ of \eqref{eq:stromnonabelian} satisfies
$$
dd^c\omega - c (F_{\theta} \wedge F_{\theta}) = 0,
$$
and therefore $x$ induces a string class
$$
[\hat H_x] \in H^3_{str}(P,\RR),
$$
where
$$
\hat H_x = p^* d^c\omega - CS(\theta).
$$

To understand the geometric meaning of the set of solutions with fixed string class, we note that a choice $[\hat H] \in H^3_{str}(P,\RR)$ determines an isomorphism class of exact Courant algebroids over $P$ (see e.g. \cite{G1annals}). More explicitly, for a choice of representative $\hat H \in [\hat H]$, the isomorphism class of exact Courant algebroids is represented by
$$
\hat E = TP \oplus T^*P,
$$
with (Dorfman) bracket
$$
[\hat X + \hat \xi,\hat Y + \hat \eta] = [\hat X, \hat Y] + L_{\hat X}\hat \eta - \iota_{\hat Y} d \hat \xi + \iota_{\hat Y}\iota_{\hat X} \hat H,
$$
and pairing
$$
\langle \hat X + \hat \xi, \hat X + \hat \xi \rangle = \hat \xi(\hat X),
$$
for vector fields $\hat X, \hat Y$ and $1$-forms $\hat \xi, \hat \eta$ on $P$. 

The exact Courant algebroid $\hat E$ comes equipped with additional structure, corresponding to the string class condition for $[\hat H]$. Firstly, we note that $[\hat H]$ is fixed by the $G$-action on $P$ -- as it always admits a $G$-invariant representative of the form \eqref{eq:hatH} -- and therefore $\hat E$ is $G$-equivariant. Secondly, $\hat E$ admits a \emph{lifted $G$-action} \cite{BuCaGu}, given by an algebra morphism $\rho\colon \mathfrak{g} \to \Omega^0(\hat E)$ making commutative the diagram
\begin{equation}
\label{eq:extendeddiagram}
\begin{gathered}
  \xymatrix{\mathfrak{g} \ar[r]^-{\rho} \ar[dr]_-{\psi} & \Omega^0(\hat E) \ar[d]^{\pi} \\  & \Omega^0(TP)}
\end{gathered}
\end{equation}
and such that the infinitesimal $\mathfrak{g}$-action on $\Omega^0(\hat E)$ induced by the Courant bracket integrates to a (right) $G$-action on $\hat E$ lifting the action on $P$. The previous data is determined up to isomorphism by the choice of real string class (see \cite[Prop. 3.7]{BarHek}).

To be more explicit, writing $\rho(z) = Y_z + \xi_z$ for $z \in \mathfrak{g}$ we have
$$
d\xi_z = i_{Y_z} \hat H
$$
Then, for a choice of connection $\theta$ on $P$ there exists a $2$-form $\hat b$ on $P$ such that
$$
\rho(z) = e^{\hat b}(Y_z - c(z,\theta \cdot)),
$$
and
$$
\hat H = p^*H - CS(\theta) + d \hat b,
$$
where $e^{\hat b} (\hat X + \hat \xi) = \hat X + \iota_{\hat X} \hat b + \hat \xi$.

Applying the general theory in \cite{BuCaGu}, the exact Courant algebroid $\hat E$ can be reduced, by means of the lifted action $\rho$, to a (transitive) Courant algebroid $E$ over the base manifold $M$, whose isomorphism class only depends on the choice of string class $[\hat H]$. Any choice of connection $\theta$ on $P$ determines an isomorphism
$$
E \cong TP/G \oplus T^*
$$
and a $3$-form $H$ on $M$, uniquely up to exact $3$-forms on $M$, such that the symmetric pairing on $E$ is given by
$$
\langle \hat X + \xi, \hat Y + \eta\rangle = \frac{1}{2}(i_{X}\eta + i_{Y}\xi) + c(\theta \hat X, \theta \hat Y),
$$
where $p \hat X = X$, and $p \hat Y = Y$, and the Dorfman Bracket is given by
\begin{equation}\label{eq:bracket}
\begin{split}
[\hat X+\xi,\hat Y+\eta] & = [\hat X, \hat Y] + L_{X}\eta - \iota_{Y}d\xi + \iota_{Y}\iota_{X}H\\
& + 2c(d_\theta(\theta \hat X),\theta \hat Y) + 2c(F_\theta (X,\cdot),\theta \hat Y) - 2c(F_\theta(Y,\cdot),\theta \hat X).
\end{split}
\end{equation}

In \cite{GRT} it is proved that solutions of the system \eqref{eq:stromnonabelian} (and hence of the Strominger system \eqref{eq:Stromingersystem2}) with fixed string class $[\hat H]$, can be recasted in terms of natural geometry in the Courant algebroid $E$ over $M$. This implies a drastic change in the symmetries of the problem: the system \eqref{eq:stromnonabelian}, with natural symmetries given by the automorphism group $\Aut P$ of $P$, is preserved by the automorphism group of the Courant algebroid $E$ (see \cite[Prop. 4.7]{GRT}) once we fix the string class. In particular, this includes $B$-field transformations for any closed $2$-form $b \in \Omega^2$ on $M$, given by
$$
\hat X \to \hat X + \iota_X b.
$$
To give the main idea, we have to go back to the physical origins of the Strominger system in supergravity, as explained in Section \ref{sec:sugra}. By Theorem \ref{th:strom2}, the system \eqref{eq:stromnonabelian} is equivalent to the Killing spinor equations 
\begin{equation}\label{eq:heterotic6dabstract}
\begin{split}
    \nabla^- \eta & = 0,\\
    (d\phi + \frac{1}{2}H) \cdot \eta & = 0,\\
    F_\theta \cdot \eta & = 0,\\
    dH - c(F_\theta \wedge F_\theta) & = 0,
\end{split}
\end{equation}
for a tuple $(g,\phi,H,\theta,\eta)$, given by a riemannian metric $g$ on $M$, a function $\phi$, a three-form $H$, a connection $\theta$ on $P$, and a non-vanishing pure spinor $\eta$ with positive chirality. 

\begin{theorem}[\cite{GRT}]\label{th:natural}
Solutions of the Killing spinor equations \eqref{eq:heterotic6dabstract} with fixed string class $[\hat H]$ are  preserved by the automorphism group of $E$.
\end{theorem}

The proof is based on the fact that a tuple $(g,H,\theta)$ satisfying the Bianchi identity $dH = c(F_\theta \wedge F_\theta)$ is equivalent to a generalized metric on $E$, while the function $\phi$ determines a specific choice of torsion-free compatible connection (which in generalized geometry is not unique). Furthermore, in \cite{Minasian,GF} it was proved in that the equations of motion \eqref{eq:motion6d-2} correspond to the Ricci and scalar flat conditions for the metric connection determined by $(g,\phi,H,\theta)$.

\subsection{The moduli problem}\label{sec:modulibis}

In this section we review on the recent progress made in the study of the moduli space of solutions for the Strominger, following \cite{GRT} (see also \cite{AGS,CGT,OssaSvanes}).

The moduli problem for the Strominger system \eqref{eq:Stromingersystem2} in dimension $n=1$ reduces to the study of moduli space of pairs $(X,\cE)$, where $X$ is an elliptic curve and $\cE$ is a polystable vector bundle over $X$, with rank $r$ and degree $0$ (see Section \ref{sec:defi}). Due to results of Atiyah and Tu (see \cite{Tu} and references therein), this moduli space corresponds to the fibred $r$-th symmetric product of the universal curve over the moduli space of elliptic curves
$$
\mathbb{H}/\operatorname{SL}(2,\ZZ),
$$
where $\mathbb{H} \subset \CC$ denotes the upper half-plane. 

For $n = 2$, Example \ref{ex:surface} shows that the moduli problem corresponds essentially to the study of tuples $(X,[\omega],\cE,\cT)$, where $X$ is a complex surface with trivial canonical bundle (a $K3$ surface or a complex torus), $[\omega]$ is a K\"ahler class on $X$, $\cE$ is a degree zero polystable holomorphic vector bundle over $X$ satisfying \eqref{eq:c2BC}, and $\cT$ is a polystable holomorphic vector bundle with the same underlying smooth bundle as $TX$. Although the moduli problem for such tuples is not fully understood even in the algebraic case, it can be tackled with classical methods of algebraic geometry and K\"ahler geometry (see \cite{AGG} and references therein).
 
As observed earlier in this work, the critical dimension for the study of the Strominger system is $n =3$. This is the lowest dimension for which the Calabi-Yau manifold $(X,\Omega)$ may be non-k\"ahlerian, and therefore new phenomena is expected to occur. To see this explicitly, we review the construction of the local moduli for the Strominger system in \cite{GRT,GRT2}. For simplicity, we will follow the setup introduced in the previous section, and deal with the system \eqref{eq:stromnonabelian}, for a six-dimensional compact spin manifold $M$. We start defining the symmetries that we will use to construct the infinitesimal moduli for \eqref{eq:stromnonabelian}. These are given by the group
$$
\Aut_0 P \subset \Aut P 
$$
where $\Aut P$ is the group of automorphism of $P$, that is, the group of $G$-equivariant diffeomorphisms of $P$, and $\Aut_0 P$ denotes the connected component of the identity. Given $g \in \Aut P$  we denote by $\check g \in \Diff(M)$ the diffeomorphism in the base that it covers. Then, $\Aut_0 P$ acts on $\cP$ by
$$
g \cdot (\Omega,\omega,\theta) = (\check g_* \Omega,\check g_*\omega,g\cdot \theta),
$$
preserving the subspace of solutions of \eqref{eq:stromnonabelian}. We define the moduli space of solutions of \eqref{eq:stromnonabelian} as the following set
$$
\mathcal{M} = \Aut_0 P \backslash \{x \in \cP: x\text{ is a solution of } \eqref{eq:stromnonabelian}\}
$$

\begin{theorem}[\cite{GRT}]
The system \eqref{eq:stromnonabelian} is elliptic.
\end{theorem}

Relying on this result, the moduli space $\mathcal{M}$ is finite-dimensional, provided that it can be endowed with a natural differentiable structure. In order to do this, a finite dimensional vector space $H^1(S^*)$ parametrizing infinitesimal variations of a solution of \eqref{eq:stromnonabelian} modulo the infinitesimal action of $\Aut_0 P$ is constructed in \cite{GRT}, using elliptic operator theory. Further, in \cite{GRT2} the Kuranishi method is applied to build a local slice to the $\Aut_0 P$-orbits in $\cP$ through a point $x \in \cP$ solving \eqref{eq:stromnonabelian}. By general theory, the local moduli space of solutions around $x$ is defined by a (typically singular) analytic subset of the slice, quotiented by the action of the isotropy group of $x$.


The construction in Section \ref{sec:gg} induces a well-defined map from the moduli space to the set of string classes
\begin{equation}
\vartheta \colon \mathcal{M} \to  H^3_{str}(P,\RR).
\end{equation}
Relying on the parallel with Maxwell theory, we call this the \emph{flux map}. We note that $H^3_{str}(P,\RR)$ is in bijection with $H^3(M,\RR)$, which corresponds to the space of infinitesimal variations of the Calabi-Yau structure $\Omega$ on $M$. Thus, potentially, restricting to the level sets of $\vartheta$ on should obtain a manifold of lower dimension (in relation to the physical problem of \emph{moduli stabilitization}). By Theorem \ref{th:natural}, each level set
$$
\vartheta^{-1}([\hat H]) \subset \mathcal{M}
$$
can be interpreted as a moduli space of solutions of the Killing spinor equations \eqref{eq:heterotic6dabstract} on the transitive Courant algebroid $E_{[\hat H]}$. On general grounds, it is expected that the moduli space $\vartheta^{-1}([\hat H])$ is related to a K\"ahler manifold, generalizing the special K\"ahler geometry in the moduli problem for polarised k\"ahlerian Calabi-Yau manifolds. Note here that $\vartheta^{-1}([\hat H])$ contains a family of moduli spaces of $\tau$-stable holomorphic vector bundles -- with varying complex structure and balanced class on $M$ --, each of them carrying a natural K\"ahler structure away from its singularities (see Section \ref{sec:gauge}).

Based on the relation with string structures, it is natural to ask which enhanced geometry can be constructed in the moduli space $\vartheta^{-1}([\hat H])$ using an integral string class
$$
[\hat H] \in H^3_{str}(P,\ZZ).
$$
This integrality condition appears naturally in the theory of $T$-duality for transitive Courant algebroids, as defined by Baraglia and Hekmati \cite{BarHek}, and it should be important for the definition of a Strominger-Yau-Zaslow version of mirror symmetry for the Strominger system \cite{Yau2005}.

\end{document}